\documentclass[a4paper,11pt,leqno]{article}
\usepackage{amssymb, amsmath, amsthm, graphicx}
\usepackage{color, epsfig, amsfonts}

\newtheorem{thm}{Theorem}[section]
\newtheorem{lem}[thm]{Lemma}
\newtheorem{cor}[thm]{Corollary}
\newtheorem{prop}[thm]{Proposition}

\newtheorem{rem}{Remark}[section]

\numberwithin{equation}{section}

\renewcommand{\a}{\alpha}
\renewcommand{\b}{\beta}
\newcommand{\e}{\varepsilon}
\newcommand{\de}{\delta}
\newcommand{\fa}{\varphi}
\newcommand{\ga}{\gamma}

\newcommand{\la}{\lambda}

\newcommand{\De}{\Delta}

\def\R{{\mathbb{R}}}

\def\Z{{\mathbb{Z}}}
\def\T{{\mathbb{T}}}

\allowdisplaybreaks

\title{Asymptotics of PDE in random environment \\
by paracontrolled calculus}

\author{Tadahisa Funaki$^\ast$, Masato Hoshino$^\star$,
Sunder Sethuraman$^\dagger$ and Bin Xie$^\ddagger$}

\date{\today } 

\begin{document}
\maketitle

\begin{abstract}
\noindent
We apply the paracontrolled calculus to study the asymptotic behavior of
a certain quasilinear PDE with smeared mild noise, which originally
appears as the space-time scaling limit of a particle system in random environment
on one dimensional discrete lattice.  We establish the convergence result and
show a local in time well-posedness of the limit stochastic PDE with spatial white noise.
It turns out that our limit stochastic PDE does not require any renormalization.
We also show a comparison theorem for the limit equation.
\end{abstract}

\footnote{ \hskip -6.5mm
$^\ast$Department of Mathematics,
School of Fundamental Science and Engineering,
Waseda University,
3-4-1 Okubo, Shinjuku-ku,
Tokyo 169-8555, Japan.
e-mail: funaki@ms.u-tokyo.ac.jp \\
$^\star$Faculty of Mathematics, Kyushu University,
744 Motooka, Nishi-ku, Fukuoka 819-0395, Japan.   \\
e-mail: hoshino@math.kyushu-u.ac.jp\\
$^\dagger$Department of Mathematics,
University of Arizona,
621 N. Santa Rita Ave.,
Tucson, AZ 85750, USA.
e-mail: sethuram@math.arizona.edu \\
$^\ddagger$Department of Mathematical Sciences,
Faculty of Science, Shinshu University,
3-1-1 Asahi, Matsumoto, Nagano 390-8621, Japan.
e-mail: bxie@shinshu-u.ac.jp; bxieuniv@outlook.com
}



\section{Introduction}

In this paper, we study the asymptotic behavior of a solution of
a certain quasilinear partial differential equation (PDE) with mild noise, which arises
in the hydrodynamic scaling limit of a microscopic interacting particle system
called zero-range process in a random environment.  We apply the method of
 the paracontrolled calculus to show that the solution converges to that of
 a stochastic partial differential equation (SPDE) with spatial white noise.
\subsection{PDE in Sinai's random environment}

Landim, Pacheco, Sethuraman and Xue \cite{LSX}
recently studied the hydrodynamic space-time scaling limit 
for the zero-range process on one dimensional discrete lattice $\Z$ 
in a random environment of Sinai's type
and derived the following PDE with mild noise for the macroscopic density
$v=v(t,x)$ of particles on $\R$:
\begin{equation}  \label{eq:1}
\partial_t v = \De \{\fa(v)\} - \nabla\{\fa(v) \dot{w}^\e(x)\}, \quad x \in \R.
\end{equation}
Here, $\De = \partial_x^2, \nabla=\partial_x$ and
$\fa(v)  \, \big(\!=E^{\nu_v}[g(\eta_0)]\big)$  is a strictly increasing
$C^2$-function of $v\ge 0$ such that $\fa(0)=0$ and $\fa'(v)\ge c$
for some $c>0$, where $\nu_v$ is a certain product measure on 
$\{0,1,2,\ldots\}^\Z$ with mean $v$ associated with the jump rate $g(k)$
of the zero-range process and $\eta_0$ denotes the particle number at 
$0$; see $\Phi(\a)$ in (3.8) of \cite{KL}, p.30.  The simplest example is 
$\fa(v)=v$ taking $g(k)=k$.  The noise is given by
\begin{equation}  \label{eq:we}
\dot{w}^\e(x) = \frac1{(a+b)\e}\big(w(x+a\e)-w(x-b\e)\big),
\end{equation}
where $\{w(x)\}_{x\in \R}$ is a two-sided Brownian motion and 
$a, b>0$; actually $a=b=1$ in \cite{LSX}.  Therefore,
$\dot{w}^\e(x) \to \dot{w}(x)$ as $\e\downarrow 0$ and the limit $\dot{w}(x)$
is the spatial white noise.  Instead of this specific form of
$\dot{w}^\e(x)$, we may take general smeared noise $\psi^\e*\dot{w}(x)$ of
$\dot{w}(x)$ defined in Subsection \ref{sec:1.3-A} or Section \ref{sec:Renorm}.

We consider the PDE \eqref{eq:1} with mild noise in more general form, assuming $\fa(v)$ 
is defined for all $v\in \R$ such that $\fa\in C^2(\R)$ and $\fa'(v)\ge c>0$,
by replacing the second $\fa(v)$ with $-\chi(v)$ such that $\chi\in C^1(\R)$:
\begin{equation}  \label{eq:2-0}
\partial_t v = \De \{\fa(v)\} + \nabla\{\chi(v) \dot{w}^\e(x)\}, 
\quad x \in \R.
\end{equation}

Our goal is to show the convergence of the solution $v=v^\e$ of
\eqref{eq:2-0} to that of the following SPDE
\begin{equation}  \label{eq:2}
\partial_t v = \De \{\fa(v)\} + \nabla\{\chi(v) \xi\}, 
\quad x \in \R,
\end{equation}
with $\xi(x)= \dot{w}(x)$.  
Roughly speaking, $\dot{w} \in C^{-\frac{d}2-} \big(:= \cap_{\de>0} C^{-\frac{d}2-\de}\big)$
if the spatial dimension is $d$ instead of $1$ so that we expect $\nabla\{\chi(v)\dot{w}\}
\in C^{-\frac{d}2-1-}$ (if $v$ is reasonable) and by Schauder estimate,
we would have $v \in C^{-\frac{d}2-1+2-} = C^{1-\frac{d}2-}$.
Thus, one can guess that \eqref{eq:2} has meaning only when $d=1$.

We will show comparison theorems for \eqref{eq:2-0} and \eqref{eq:2},
by which one can deal with $\fa$ and $\chi$ defined only for $v\ge 0$
if $\chi$ satisfies $\chi(0)=0$, see 
Corollary \ref{cor:1.3} and Lemma \ref{lem:comparison}.

If we drop $\nabla$ in the second term of \eqref{eq:2},
it is an equation known as the generalized parabolic Anderson model
and studied in the framework of singular SPDEs,
\cite{GIP-15}, \cite{Hai-14}, \cite{HaLa-15}, \cite{HaLa-18}, \cite{BDH-19}, \cite{FuMa-19}.
Note that the equation \eqref{eq:2}, with both $\fa, \chi$ being linear
and $\nabla$ dropped, is originally called the parabolic Anderson model.
The generalized parabolic Anderson model has meaning for the spatial dimension $d\le 3$.

\subsection{Integrated quasilinear stochastic PDE}

We actually study, instead of \eqref{eq:2},  the equation
\begin{equation}  \label{eq:Q-ab}
\partial_t u = a(\nabla u)\De u + g(\nabla u)\cdot\xi,
\end{equation}
where $a(v)=\fa'(v)$, $g(v)=\chi(v), v \in \R$, and $\xi$ is the spatial white noise.
If we set $v:=\nabla u$, then we can recover the equation
\eqref{eq:2}:
\begin{align*}
\partial_t v &= \nabla \big( a(v)\nabla v\big) 
+ \nabla \big( g(v)\cdot\xi \big) \\
&= \De\{\fa(v)\} + \nabla\{\chi(v) \cdot\xi\}.
\end{align*}
In other words, \eqref{eq:Q-ab} is an integrated form of
\eqref{eq:2}.  The relation between the equations \eqref{eq:2}
and \eqref{eq:Q-ab} is similar to that of stochastic Burgers equation
and KPZ equation.

We work with the equation \eqref{eq:Q-ab} on $\T = [0,1]$
with periodic boundary condition following the method of
Bailleul, Debussche and Hofmanov\'{a} \cite{BDH-19}, in which they
studied the case $a=a(u), g=g(u)$ on $\T^2$ instead of $a=a(\nabla u), g=g(\nabla u)$
on $\T$ in our case.  Roughly, the noise term behaves as $g(\nabla u)\cdot\xi \in C^{\frac12-}
\times C^{-\frac12-}$ in our case, while $g(u)\cdot\xi \in C^{1-}\times C^{-1-}$
in \cite{BDH-19}.

The equation \eqref{eq:2} considered on $\T$ has a mass conservation law:
$$
\int_\T v(t,x)dx = m
$$
for all $t\ge 0$ with a constant $m\in\R$.  This is caused by the conservation of
particle number of the underlying microscopic system.  Since $v=\nabla u\ge 0$
for the original model, $u$ should be increasing in $x$ and
$u(t,1)=u(t,0)+m$ holds with some $m>0$.  Therefore, it is more natural to consider
\eqref{eq:Q-ab} under the modified periodic condition:
\begin{equation} \label{eq:mp}
u(t,x+n) = u(t,x) + nm, \quad n\in \Z,\; x \in \R.
\end{equation}
Indeed, to consider \eqref{eq:Q-ab} under the condition \eqref{eq:mp}, set
$\bar u(t,x) := u(t,x) - mx$, then $\bar u(t,x)$ satisfies
the usual periodic boundary condition $\bar u(t,x+1) = \bar u(t,x)$
and the SPDE
$$
\partial_t \bar u = a(\nabla\bar u + m) \De \bar u
+ g(\nabla\bar u+m)\cdot \xi.
$$
Therefore, instead of \eqref{eq:mp}, we may consider \eqref{eq:Q-ab} under
the usual periodic boundary condition with $a(v), g(v)$ replaced by
$a(v+ m), g(v+ m)$, respectively.

The equation \eqref{eq:Q-ab} has a property that, if $u$ is a solution,
then $u+c$ is also a solution for every $c\in \R$.  In other words,
\eqref{eq:Q-ab} is an equation for the shape of $u$ and its graph is
invariant under the vertical shift.  Or, \eqref{eq:Q-ab} is essentially an
equation for the slope $\nabla u$ of $u$.  This is close to the
$\nabla\phi$-interface model \cite{FS}, though its driving force
is the space-time noise.  This gives a clear difference
from the equation considered in \cite{BDH-19}.  One can expect
that the structure of invariant measures for the system on $\R$ 
would be quite different.

The equation \eqref{eq:Q-ab} with $a(v)=\frac12$ and $g(v)=v$ can be 
interpreted as the Kolmogorov equation associated with the stochastic
differential equation (SDE)
\begin{equation}  \label{eq:SDE}
dX_t = b(X_t)dt + dW_t,
\end{equation}
where $b(x)=\xi(x)$ and $W_t$ is a one dimensional Brownian motion,
see \cite{FIR}.  The SDE \eqref{eq:SDE} is studied by \cite{CC} for $b\in C^\b$
with $\b\in (-\frac23,-\frac12]$ and this covers the one dimensional 
spatial white noise as in our paper.  The process $X_t$ determined by
the SDE \eqref{eq:SDE} with $b=\xi$, the spatial white noise independent
of $W$, is called the Brox diffusion, cf.\ \cite{LSX}.

\subsection{Main results and structure of the paper}  \label{sec:1.3-A}

Let $C^\a, \a \in \R$ and $\mathcal{L}_T^\a, \a \in (0,2)$ be the spatial and parabolic H\"{o}lder spaces, respectively, explained
in Subsection \ref{sec:2.1}.  Let $\xi$ be the spatial white noise on $\T$.
We call $\xi^\e,\e>0$, the smeared noise of $\xi$ if it is defined by
$\xi^\e=\psi^\e*\xi$, where $\psi^\e(x) = \frac1\e \psi\big(\frac{x}\e\big)$
and the mollifier $\psi$ is a measurable and integrable function on
$\R$ with compact support satisfying $\int_\R\psi(x)dx=1$, see Lemma \ref{lem:2.9}.
Note that the noise $\dot{w}^\e$ in \eqref{eq:we} considered on $\T$ 
is a smeared noise of the spatial white noise $\xi=\dot{w}$ in this sense, see Remark
\ref{rem:5.1}.

Our main result is formulated as follows.  In our case, we do not need any renormalization.

\begin{thm}  \label{thm:main}
Assume that $a\in C_b^3(\R)$ satisfies $c\le a(v)\le C$
for some $c, C>0$ and $g\in C_b^3(\R)$.  Let an initial value
$u_0\in C^\a$ with $\a\in (\frac43, \frac32)$ be given.  Then, there
exists a random time $T>0$ defined on the same probability space
as $\xi$ such that the solutions $u^\e$ of the SPDE
\begin{equation*}  
\partial_t u^\e = a(\nabla u^\e)\De u^\e + g(\nabla u^\e)\cdot\xi^\e,
\quad u^\e(0)=u_0,
\end{equation*}
with the smeared noise $\xi^\e$ of $\xi$ converge in probability
in $\mathcal{L}_T^\a$ as $\e\downarrow 0$ to $u\in \mathcal{L}_T^\a$,
which is a unique solution up to the time $T$ of the SPDE
\eqref{eq:Q-ab} on $\T$ defined in paracontrolled sense, that is,
satisfying the condition \eqref{eq:5.5} below, with $u(0)=u_0$.
In particular, the limit $u$ is independent of the choice of the mollifier $\psi$.
\end{thm}

When we consider \eqref{eq:Q-ab} under the modified periodic condition
\eqref{eq:mp} with $m>0$ assuming $\nabla u(0,x)\ge 0$, we can expect that
$\nabla u(t,x)\ge 0$ holds for the solution of \eqref{eq:Q-ab} for $t>0$.
We will show this for $v=\nabla u$ by considering \eqref{eq:2} with smooth
noise $\xi^\e$, see Lemma \ref{lem:comparison}.  Then, by 
Theorem \ref{thm:main}, we can find a subsequence of $\e\downarrow 0$
along which $u^\e$ converge to $u$ almost surely and, taking the limit,
we obtain the following corollary.

\begin{cor}\label{cor:1.3}
Assume $\fa\in C^4([0,\infty))$, $\chi\in C^3([0,\infty))$ and
they satisfy the conditions $\chi(0)=0$,
$c\le \fa'(v)\le C$ and $|\chi'(v)| \le C\fa'(v)$ for some
$c, C>0$.  Then, for the solution $u(t)$ of 
the paracontrolled SPDE \eqref{eq:Q-ab}, if $\nabla u(0,x)\ge 0$
holds for all $x\in \T$, we have $\nabla u(t,x)\ge 0$ for all $0\le t \le T$
and $x\in \T$.
\end{cor}

In Section \ref{sec-2.1-0307}, we rewrite the SPDE \eqref{eq:Q-ab}
into a certain fixed point problem in a setting of the
paracontrolled calculus and we solve it in Section \ref{sec-2.2-0307},
see Theorem \ref{thm-2.1-0304}.  Theorem \ref{thm:main} follows from
Theorem \ref{thm-2.1-0304} and Lemma \ref{lem:2.9}.
Section \ref{sec-4-0406} is devoted to the proof of Lipschitz estimate
on $\zeta$, which appears as a remainder term in the fixed point problem.
In Section \ref{sec:Renorm}, we study the convergence in $C^{2\a-3}$
of the resonant term $\Pi(\nabla X,\xi)$, which is a quadratic function 
of the spatial white noise $\xi$, where $X= (-\De)^{-1}Q\xi$, 
$Q\xi= \xi-\xi(\T)$ and $\a$ is as in Theorem \ref{thm:main}.
The reason why our equation does not require
renormalization lies in the fact that this term contains the derivative
$\nabla X$.  This is different from  \cite{BDH-19} and more close to
the stochastic Burgers equation.
In Section \ref{sec:comp}, we show a comparison theorem for
the SPDE \eqref{eq:2} with smooth noise and this leads to Corollary 
\ref{cor:1.3}.

\section{Reduction to fixed point problem} \label{sec-2.1-0307}

The main purpose of this section is to reduce the  SPDE \eqref{eq:Q-ab} 
driven by one dimensional spatial white noise $\xi$ on $\T$ to a fixed 
point problem for the map $\Phi$ defined by \eqref{eq-2.4-0304} by
means of  the  paracontrolled calculus. 

\subsection{Function spaces, regularity exponents, paraproduct and variable $X$}
\label{sec:2.1}

Let us first introduce several function spaces.  As in \cite{GIP-15}, \cite{BDH-19}, 
for regularity exponent $\a\in \R$, we denote by $C^\a\equiv C^\a(\T) 
:= B_{\infty,\infty}^\a(\T,\R)$ the spatial H\"{o}lder-Besov space 
equipped with the norm $\|\cdot\|_{C^\a}$.  For $T>0$, we  write $C_TC^\a =C([0,T],C^\a)$
for the space of $C^\a$-valued continuous functions on $[0,T]$ endowed with the supremum norm 
$\|\cdot\|_{C_TC^\a}$ and write $C_T^{\a}L^\infty = C^\a([0,T],L^\infty)$ with
$\a \in (0,1)$ 
for the space of $\a$-H\"{o}lder continuous functions from $[0,T]$ to 
$L^\infty$ with the seminorm $\|f\|_{C_T^{\a}L^\infty} =\sup_{0\leq s\neq t\leq T}\|f(t)- f(s)\|_{L^\infty}/|t-s|^\alpha$ for $f\in C_T^{\a}L^\infty$,
 where $L^\infty= L^\infty(\T)$. For $\a \in (0,2)$, we also define 
$\mathcal{L}_T^\a := C_TC^\a \cap C_T^{\a/2}L^\infty$, equipped with
the norm
\begin{align*}
\|\cdot\|_{\mathcal{L}_T^\a} =\|\cdot\|_{C_TC^\a} + \|\cdot\|_{C_T^{\frac{\a}{2}}L^\infty}.
\end{align*}

In the following, we will use three regularity exponents $\a, \b$ and $\ga$.
In particular, $\a$ is the exponent for the solution $u$ of \eqref{eq:Q-ab} and
$\b$ is that for a certain function of its derivative $\nabla u$ (see $u'$ below) so that $\b<\a-1$.
We expect $\a= \frac32-\de$ for every $\de>0$, since $\xi\in C^{-\frac12-\de}$ 
for one dimensional spatial white noise would imply $u\in C^{\frac32-\de}$ by Schauder estimate.

More precisely, we assume the following conditions for these three exponents.
For $\a$, we assume that 
$\alpha \in (\frac43, \frac{3}{2})$ and consider  $\xi \in C^{\alpha-2}$.  
For $\b$, noting that  $\frac13<\alpha-1$ for $\a$ in this interval,  
we  take  $\beta \in (\frac13,\alpha-1)$. Then, we have that $3-2\a <\frac13$ and  $2\a+ \b -3>0.$
The third regularity exponent $\gamma$ is taken as $\gamma \in (2\b +1, \a +\b)$,
which will be mainly used in Section \ref{sec-2.2-0307}.  In particular, we have 
$\ga + \b -2>0.$ Throughout this paper, unless otherwise noted, we will assume
$\a, \b$ and $\ga$ satisfy the aforementioned conditions. 

As in \cite{BDH-19}, we denote by $\Pi_fg ( = f\prec g \  \text{in}\ 
 \cite{GIP-15})$ the paraproduct,  by
$\Pi(f,g) ( = f\circ g)$ the resonant term  and by $\bar\Pi_fg (= f\prec\!\!\!\prec g)$ the modified paraproduct, respectively. Then, for two distributions $f$ and 
$g$, the general product $fg$ can be (at least formally) written as
$fg=\Pi_f g + \Pi(f,g) + \Pi_gf$, which is called the Littlewood-Paley
decomposition.  See Section 2.1 and (36) of \cite{GIP-15} or p.\ 43 and p.\ 45 of 
\cite{BDH-19} for precise definitions of these notion.

Define $X=(-\De)^{-1}Q\xi$, more precisely, a zero spatial mean solution of 
\begin{equation}  \label{eq:Xxi}
-\De X = Q\xi,
\end{equation}
where $Q\xi:= \xi-\xi(\T)$ and $\xi(\T) \equiv \hat{\xi}(0) = \int_\T \xi(x)dx$,
see also \eqref{eq:4.X} below. Consider $(u,u') \in \mathcal{L}_T^\a\times \mathcal{L}_T^\b$.
We call $(u,u')$ is paracontrolled by $X$ if
\begin{equation}  \label{eq:5.5}
u= \bar\Pi_{u'} X+ u^\sharp, \quad u^\sharp \in \mathcal{L}_T^\a
\end{equation}
holds with 
\begin{equation}  \label{eq:2-abc}
\|(u,u')\|_{\a,\b,\ga} := \|u'\|_{\mathcal{L}_T^\b} + \|u^\sharp\|_{\mathcal{L}_T^\a}
+ \sup_{0<t\le T} t^{\frac{\ga-\a}2} \|u^\sharp(t)\|_{C^{\ga}} < \infty.
\end{equation}
We denote by ${{\bf C}_{\a, \b, \gamma}}(X)$ the space of all functions
$(u, u')$  
paracontrolled by $X$. Actually, we introduce $(u, u')$ expecting 
$u$ to be the solution of  \eqref{eq:Q-ab} and $u'=\frac{g(\nabla u)}
{a(\nabla u)}$.

\subsection{Basic estimates in paracontrolled calculus}

Before starting to transform the equation \eqref{eq:Q-ab} into a fixed point problem,
let us prepare some fundamental estimates in the paracontrolled 
calculus, which will be frequently used in the following.

Let us first summarize some known results for 
the H\"{o}lder-Besov space $C^\alpha$, which are mentioned in
Lemma 2.1 and  p. 62 of \cite{GIP-15}.  Throughout this paper, 
we write $a\lesssim b$ for two non-negative functions  $a$ and $b$
 (mostly norms of distributions or their products) if there exists a positive constant $C$ 
 independent of the variables under consideration such that $a \leq C b$.

\begin{lem}\label{lem-5.1-0301}
{\rm (i)} For  $\a \leq \b$,  we have $\|\cdot\|_{C^{\a}} \leq 
\|\cdot\|_{C^{\b}}$. Moreover, $\|\cdot \|_{L^\infty} \lesssim \|\cdot\|_{C^{\a}}$ for  
$\alpha >0$ and conversely $\|\cdot\|_{C^{\a}} \lesssim \|\cdot \|_{L^\infty}$ for  $\alpha \leq 0$.
\\
{\rm (ii)} (Bony's estimate)
The following hold.\\
\vspace{-8mm}
\begin{itemize}
 \setlength\itemsep{0em}
\item For {$\alpha > 0$} and  $\b\in\R$, $\|\Pi_u v\|_{C^\beta}\lesssim \|u\|_{L^\infty} \|v\|_{C^\b}$.  
\item For  $\alpha\not= 0$ and $\b\in\R$, $\|\Pi_u v\|_{C^{(\a\wedge 0)+ \b}} \lesssim \|u \|_{C^\a} \|v\|_{C^\b}$.
\item For  $\alpha +\b> 0$, $\|\Pi(u, v)\|_{C^{\a+ \b}} 
\lesssim \|u\|_{C^\a} \|v\|_{C^\b}$.
\end{itemize} 
\vspace{-3mm}
In particular, the product $uv$ is well-defined if and only if  $\a +\b >0$, in this case  
$uv \in  C^{\a\wedge\b}$ and 
\begin{align*}
\|uv\|_{C^{\a\wedge\b}} \lesssim \|u\|_{C^\a} \|v\|_{C^\b}
\end{align*}
holds for $\a \b\neq 0$. 
\end{lem}
The next is an important result for  the  commutator.
\begin{lem} (Lemma 2.4  \cite{GIP-15}, Commutator lemma) 
\label{lem-2.2-0317}
Let $\a \in (0,1)$ and $\b, \gamma \in \R$ satisfy the conditions of $\b +\gamma<0$ and $\a+\b +\gamma>0$. For any smooth functions $u,v,w$, define
$C(u, v, w)$ by
$$
C(u, v, w) := \Pi(\Pi_u v, w) -u\Pi(v,w).
$$
Then, $C(u, v, w)$ is uniquely extended to a bounded trilinear operator from $C^\a \times C^\b \times C^{\gamma}$ to $C^{\a +\b +\gamma}$ and 
$$
\|C(u, v, w)\|_{C^{\a +\b +\gamma}} \lesssim 
\|u\|_{C^\a}\|v\|_{C^\b} \|w\|_{C^\gamma}.
$$
\end{lem}

Since $\nabla u$ is included in the coefficients of  the SPDE
\eqref{eq:Q-ab}, we need to  study the regularity of $\nabla u$
whenever $u\in \mathcal{L}_T^{\a}$.   
The next lemma shows that $\nabla u \in  \mathcal{L}_T^{\a -1}$ holds by  an interpolation theorem.

\begin{lem}\label{lem-3.3-0318}
Let $\alpha\in(1,2)$.  Then,
for any $u\in \mathcal{L}_T^\a$, we have  $\nabla u \in  \mathcal{L}_T^{\a -1}$ and 
\begin{align}\label{eq-3.1-0318}
\|\nabla u\|_{\mathcal{L}_T^{\a -1}}  
\lesssim \|u\|_{\mathcal{L}_T^\a}.
\end{align}
\end{lem}
\begin{proof}
From  the fact that  the operator $\nabla: C^\a \mapsto  C^{\a-1}$ is continuous, see Proposition 2.3 of \cite{Ho} or Lemma A.1 and its remark in  \cite{GIP-15}, it follows that 
for all $u\in C^\a$, $\|\nabla u\|_{C^{\a-1}} \lesssim \|u\|_{C^\a}$.
Then, we can easily show  that  
$\nabla u \in C_TC^{\a-1} \cap C_T^{\a/2}C^{-1}$.
More precisely, we have
\begin{align} \label{eq-2.7-04031}
\|\nabla u\|_{C_TC^{\a-1} } \lesssim \|u\|_{C_TC^{\a} } \ \ 
\text{and} \ \ 
\|\nabla u\|_{C_T^{\a/2}C^{-1}} \lesssim \|u\|_{C_T^{\a/2}L^\infty}. 
\end{align}
The first part is clear. Noting that $u\in C_T^{\a/2}L^\infty$ and the continuous embedding 
of $L^\infty$ in $C^0$ (see Lemma \ref{lem-5.1-0301}-{\rm(i)}), we have 
\begin{align*}
\|\nabla u(t)- \nabla u(s)\|_{C^{-1}} \lesssim \| u(t)-  u (s)\|_{C^0} \lesssim \| u(t)-  u (s)\|_{L^\infty}, 
\end{align*}
which implies that $\|\nabla u\|_{C_T^{\a/2}C^{-1}} \lesssim \|u\|_{C_T^{\a/2}L^\infty }$ and in particular $\nabla u\in C_T^{\a/2}C^{ -1}$. 

Now it is enough for us to show 
$\|\nabla u\|_{C_T^{\frac{\a -1}2}L^\infty} \lesssim  \|u\|_{\mathcal{L}_T^\alpha}$.
Noting that $B_{\infty,1}^0\subset L^\infty$ and 
using Theorem 2.80 of \cite{BCD} (see Remark \ref{rem-2.1-0403} below for details), we have  that 
\begin{align} \label{eq-2.7-0403}
\|v \|_{L^\infty} \lesssim   \|v\|_{B_{\infty,1}^0} 
 \lesssim 
\|v\|_{C^{-1}}^{\frac{\a -1}{\a}} 
 \|v\|_{C^{\a -1}}^{\frac{1}{\a}}
\end{align}
holds for any $v \in C^{\a -1}\cap C^{-1}$. 

Applying \eqref{eq-2.7-0403} to $\nabla u(t)- \nabla u(s)$ and then using \eqref{eq-2.7-04031}, we  obtain that 
\begin{align}\label{eq-3.3-0318}
 \|\nabla u\|_{C_T^{\frac{\a -1}2 }L^\infty} 
\lesssim& \sup_{0\leq s< t\leq T}\frac{\|\nabla u(t) -\nabla u(s) \|_{C^{-1}}^{\frac{\a -1}{\a}} 
 \|\nabla u(t) -\nabla u(s) \|_{C^{\a -1}}^{\frac{1}{\a}}}{|t-s|^{\frac{\a-1}2} } \\
 \lesssim & \left( \|\nabla u\|_{C_T^{\frac{\a }2 }C^{-1}}\right)^{\frac{\a-1}{\a}}  \|\nabla u\|_{C_TC^{\a-1}}^\frac{1}{\a}  \notag 
\\
\lesssim &   \left( \| u\|_{C_T^{\frac{\a }2 }L^\infty }\right)^{\frac{\a -1}{\a}} \| u\|_{C_TC^{\a}}^\frac{1}{\a} \notag  \\
\le& \|u\|_{\mathcal{L}_T^\alpha}.\notag 
\end{align}
Consequently, combining  \eqref{eq-2.7-04031} with \eqref{eq-3.3-0318}, we  obtain \eqref{eq-3.1-0318}.
\end{proof}

\begin{rem} \label{rem-2.1-0403}
To show \eqref{eq-2.7-0403}, we utilize
the fact $B_{\infty,\infty}^\alpha=C^\alpha$ and the second part of Theorem 2.80 of \cite{BCD}, the interpolation  inequality  for general nonhomogeneous Besov spaces:  there exists
a constant $C>0$ such that  
for $\a_1<\a_2$, $\theta\in(0,1)$, and $p\in[1,\infty]$,
$$
\|u\|_{B_{p,1}^{\theta \a_1+(1-\theta)\a_2}} \leq \frac{C}{(\a_2 -\a_1)\theta(1-\theta)}
\|u\|_{B_{p,\infty}^{\a_1}}^\theta \|u\|_{B_{p,\infty}^{\a_2}}^{1-\theta},
$$
where $B_{p,1}^\a$ and $B_{p,\infty}^\a$ denote nonhomogeneous Besov spaces, see
Definition 2.68 of \cite{BCD} for details.  This interpolation  inequality is originally stated for  the spaces on $\mathbb{R}$ in
\cite{BCD}. However, by similar arguments, we can show the same inequality 
for the spaces on $\mathbb{T}$.

A simple interpolation theorem for the spaces $C^\a$ implies
$$
\|u\|_{C^{\theta \a_1+(1-\theta)\a_2}} \leq  \|u\|_{C^{\a_1}}^\theta \|u\|_{C^{\a_2}}^{1-\theta},
$$
for $\a_1, \a_2\in \R$ and $\theta\in(0,1)$, see Proposition 2.1 of \cite{Ho} or the first part
 of Theorem 2.80 of \cite{BCD}.  By using this inequality, due to a little gap between two spaces
$L^\infty$ and $C^0$ as in Lemma \ref{lem-5.1-0301}-{\rm(i)}, 
we can  show a weaker result than \eqref{eq-3.1-0318}, that is,
for all $ \b \in (\frac{1}3, \a-1)$, $\|\nabla u\|_{\mathcal{L}_T^\b}  
\lesssim \|u\|_{\mathcal{L}_T^\a}$
with implicit constants $C=C(\a, \b)$, without relying on
general nonhomogeneous Besov spaces.
\end{rem}

We summarize estimates for the modified paraproduct $\bar\Pi$.
Note that the first estimate \eqref{eq-2.3-0424} in the following lemma is similar
to that for $\Pi_uv$ in Lemma \ref{lem-5.1-0301}-(ii), but, since the definition of
the modified paraproduct involves time integral, the estimate is given by
uniform norms in time.

\begin{lem}
\label{lem:bar}
Let $\b \in \R$. Then the following hold.\\
{\rm (i)} If  $\a \neq 0$, then for any $u\in C_TC^\a$ and $v\in C_TC^\b$, we have 
\begin{align}  \label{eq-2.3-0424}
\| \bar \Pi_uv \|_{C_TC^{(\a \wedge 0)+ \beta} }\lesssim
\|u\|_{C_TC^\a}\|v\|_{C_TC^\beta}.
\end{align}
{\rm (ii)} If $\alpha\in(0,2)$, then for any $u\in \mathcal{L}_T^\a$
 and $v\in C_TC^\b$, we have  
\begin{align}  \label{eq-2.4-0424}
\|\bar\Pi_uv-\Pi_uv\|_{C_TC^{\alpha+\beta}}\lesssim\|u\|_{\mathcal{L}_T^\alpha}\|v\|_{C_TC^\beta}.
\end{align} 
\end{lem}
Lemma \ref{lem:bar}-{\rm (i)} follows obviously from  Lemma 2.7 of \cite{GP-17} and  $\|u\|_{C_TL^\infty} \lesssim \|u\|_{C_TC^\a}$ for 
$\a>0$.
Lemma \ref{lem:bar}-{\rm (ii)} is taken from  Lemma 2.8 of \cite{GP-17} and  the special case with $\a \in (0,1)$ is proved in Lemma 5.1 of  \cite{GIP-15}.

Combining Lemma \ref{lem:bar}-{\rm (ii)} with Lemma \ref{lem-2.2-0317}, we easily have the following modified version of commutator lemma.

\begin{lem}
\label{lem modified commutator}
Let $\a \in (0,1)$ and $\b, \gamma \in \R$ satisfy the same conditions as Lemma \ref{lem-2.2-0317}. For any smooth functions $u,v,w$ on $[0,T]\times\mathbb{T}$, define 
$\bar{C}(u, v, w)$ by
$$
\bar{C}(u, v, w) := \Pi(\bar\Pi_u v, w) -u\Pi(v,w).
$$
Then, $\bar{C}(u, v, w)$ is uniquely extended to a bounded trilinear operator from $\mathcal{L}_T^\a \times C_TC^\b \times C_TC^{\gamma}$ to $C_TC^{\a +\b +\gamma}$ and 
$$
\|\bar{C}(u, v, w)\|_{C_TC^{\a +\b +\gamma}} \lesssim 
\|u\|_{\mathcal{L}_T^\a}\|v\|_{C_TC^\b} \|w\|_{C_TC^\gamma}. 
$$
\end{lem}

\begin{proof}
The result is obvious from the inequalities
$$
\|C(u, v, w)\|_{C_TC^{\a +\b +\gamma}} \lesssim 
\|u\|_{C_TC^\a}\|v\|_{C_TC^\b} 
 \|w\|_{C_TC^\gamma} 
$$
by Lemma \ref{lem-2.2-0317} and
\begin{align*}
\|\bar{C}(u,v,w)-C(u, v, w)\|_{C_TC^{\a +\b +\gamma}}
&=\|\Pi(\bar\Pi_uv-\Pi_uv,w)\|_{C_TC^{\a+\b+\gamma}}\\
&\lesssim 
\|u\|_{\mathcal{L}_T^\a}\|v\|_{C_TC^\b} 
 \|w \|_{C_TC^\gamma} 
\end{align*}
by Lemma \ref{lem-5.1-0301}-(ii) and Lemma \ref{lem:bar}-{\rm (ii)}.
\end{proof}

The next lemma gives estimates on the commutators $[\nabla, \bar{\Pi}]$ and 
$[\De, \bar{\Pi}]$ and follows from  Lemma \ref{lem:bar}-{\rm (i)},
cf.\ Lemma 5.1 of \cite{GIP-15}.

\begin{lem}  \label{lem:37}
Let $T>0, \a\in (0,1), \b\in\R$ and let $u\in C_TC^\a$ and $v\in C_TC^\b$.
Then, the following hold.
\begin{align}
& \|[\nabla,\bar\Pi_{u}]v \|_{C_TC^{\a +\b -1}}\lesssim \|u\|_{ C_TC^\a } \|v\|_{C_TC^\b}, \label{eq-4.1-0406}
\\
& 
\|[\De,\bar\Pi_{u}]v \|_{C_TC^{\a +\b -2}}\lesssim \|u\|_{ C_TC^\a} \|v\|_{C_TC^\b}. \label{eq-4.2-0406}
\end{align}
\end{lem}

\begin{proof}
Since the proofs for \eqref{eq-4.1-0406}  and \eqref{eq-4.2-0406}
are essentially same, we only give the proof of  \eqref{eq-4.1-0406},
note that Lemma \ref{lem-3.3-0318} is easily generalized to $\De u$.
By the definition of $\bar\Pi_{u}v$, we have
$$
[\nabla,\bar\Pi_{u}]v \equiv \nabla\bar\Pi_{u}v - \bar\Pi_{u}\nabla v
= \bar\Pi_{\nabla u}v.
$$ 
Therefore, noting $\a-1<0$ 
and then applying Lemma \ref{lem:bar}-{\rm (i)} together with the fact that $\|\nabla u\|_{C^{\a-1}}\lesssim \|u\|_{C^\a}$, we have
\begin{align*}
\|\bar\Pi_{\nabla u}v\|_{C_TC^{\a +\b -1}}
& \lesssim \|\nabla u\|_{C_TC^{\a-1} } \|v\|_{C_TC^\b}\\
& \lesssim \|u\|_{C_TC^{\a} } \|v\|_{C_TC^\b}
\end{align*}
and this shows the conclusion.
\end{proof}
We also need the associativity for the modified paraproduct $\bar{\Pi}$. To show it, 
we first state the associative property for the paraproduct. 
\begin{lem}(Lemma 2.6 \cite{GP-17}) \label{lem-2.3-0405}
If $\a>0$ and $\b \in \R$, then 
$$ 
\|\Pi_u ( \Pi_v w) -\Pi_{uv }w \|_{C^{\a +\b}} 
\lesssim\|u \|_{C^\a}\|v\|_{C^\a}\|w\|_{C^\b}
$$
holds for all $u, v \in C^{\a}$ and $w\in C^{\b}$.
\end{lem}
Based on the above  associative property, we can show that the associativity holds for the modified paraproduct $\bar{\Pi}$.
\begin{lem}\label{lem-2.4-0405}
Let $\a\in (0,1)$ and $\b \in \R$, and let us define 
$$
R(u, v; w) =\Pi_u( \bar{\Pi}_v w) -\Pi_{uv}w
$$
for $u \in C_TC^\alpha$, $v \in\mathcal{L}_T^\alpha$ 
and $w \in C^{\b}$.
Then we have 
$$
\|R(u, v; w)\|_{C_TC^{\a +\b}} \lesssim
\|u\|_{C_TC^\a }\|v\|_{\mathcal{L}_T^\a}\|w\|_{C^\b}.
$$

\end{lem}
\begin{proof}
Noting $\a\in (0,1)$ and then applying Lemmas \ref{lem-5.1-0301},    \ref{lem-2.3-0405} and \eqref{eq-2.4-0424} in Lemma \ref{lem:bar},
we deduce that 
\begin{align*}
& \|R(u, v; w)(t)\|_{C^{\a +\b}}   \\
 \leq & \|\big( \Pi_u ( \bar{\Pi}_v w) - \Pi_u ( {\Pi}_v w) \big)(t)\|_{C^{\a +\b}}
+
\|\big( \Pi_u ( {\Pi}_v w)  -\Pi_{uv }w \big)(t)\|_{C^{\a +\b}} \\
\lesssim & \|u(t)\|_{C^\a} \|\big(\bar{\Pi}_v w -{\Pi}_v w\big)(t)\|_{C^{\a+\b}} + \|u(t)\|_{C^\a}\|v(t)\|_{C^\a}\|w\|_{C^\b} \\
\lesssim & \|u\|_{C_TC^\a }\|v\|_{\mathcal{L}_T^\a}\|w\|_{C^\b}
\end{align*}
for all $t\in [0,T]$, which gives the result.
\end{proof}

In the end, we give some estimates for nonlinear functions of $\nabla u$.  These 
estimates are mainly used in Sections \ref{sec-2.2-0307} and \ref{sec-4-0406},
where we will take  $F=a, a' , g$ or $g'$ and choose a proper exponent $\a$ 
according to the situation; recall that we assume that $a, g \in C^3_b(\R)$.

\begin{lem} 
\label{lem-2.8-0424}
Let $F\in C_b^2(\R)$ and $\a \in (1,2)$. Then the following hold.
\\
{\rm (i)} For any 
$u \in C^\a$,  we have
\begin{align}\label{eq-2.11-0424}
 \|F(\nabla u)\|_{C^{\alpha-1}}\lesssim\|F\|_{C^1}(1+\|u\|_{C^\alpha}),
\end{align}
and for any $u, v \in C^\a$, we have the local Lipschitz estimate
\begin{align} \label{eq-2.12-0424}
 \|F(\nabla u) -F(\nabla v)\|_{C^{\alpha-1}}\lesssim\|F\|_{C^2}(1+\|u\|_{C^\alpha})\|u-v\|_{C^\a}.
\end{align}
{\rm (ii)} Let $0<\b \leq \a-1$. Then, for any 
$u \in \mathcal{L}_T^\a$,  we have
\begin{align} \label{eq-2.13-0424}
 \|F(\nabla u)\|_{\mathcal{L}_T^\b}\lesssim
T^{\frac{\a-\b-1}{2}} \|F\|_{C^1}(1+\|u\|_{\mathcal{L}_T^\a} ) +\|F(\nabla u(0))\|_{C^\b}.
\end{align}
For any $u, v \in \mathcal{L}_T^\a$ satisfying $u(0)=v(0)$, we have the local Lipschitz estimate
\begin{align}\label{eq-2.14-0424}
 \|F(\nabla u) -F(\nabla v)\|_{\mathcal{L}_T^\b }\lesssim
T^{\frac{\a-\b-1}{2}} \|F\|_{C^2}(1+\|u\|_{\mathcal{L}_T^\a})
\|u-v\|_{\mathcal{L}_T^\a}.
\end{align}
\end{lem}
\begin{proof} 
Using Lemma 9-$1$ of \cite{BDH-19}, we have for any 
$u\in C^\a$
\begin{align*}
& \|F(\nabla u)\|_{C^{\alpha-1}}\lesssim\|F\|_{C^1}(1+\|\nabla u\|_{C^{\alpha-1}} ). 
\end{align*}
Then, recalling the continuity of the operator $\nabla: C^\a \mapsto  C^{\a-1}$, we have the desired result \eqref{eq-2.11-0424}. 
We can show \eqref{eq-2.12-0424} by analogous arguments.

As for the results in {\rm (ii)}, they  can be shown by combining Lemma 9-$2$ of \cite{BDH-19} and Lemma \ref{lem-3.3-0318}. 
\end{proof}

\subsection{Derivation of fixed point problem}

Let $\alpha,\beta,\gamma$ be the exponents given in Subsection \ref{sec:2.1}. 
Let us now reformulate the equation \eqref{eq:Q-ab} into a fixed point problem
for the map $\Phi$ on ${\bf C}_{\a, \b, \ga}(X)$ defined as follows:
\begin{align} \label{eq-2.4-0304}
& \Phi(u,u') := (v,v')
\intertext{and}
& v'= \frac{g(\nabla u) - \big(a(\nabla u)-a(\nabla u_0^T)\big) u'}{a(\nabla u_0^T)}, \label{eq-2.7-0323}  \\
\label{eq-2.8-0323}& \mathcal{L}^0v= \Pi_{a(\nabla u_0^T)v'} \xi
+g'(\nabla u)\Pi(\nabla u^\sharp,\xi)-a'(\nabla u)\Pi(\nabla u^\sharp,\bar\Pi_{u'}\xi)\\
\notag&\qquad+\big(a(\nabla u)-a(\nabla u_0^T)\big)\De u^\sharp
+\zeta, 
\end{align}
where $\mathcal{L}^0 := \partial_t -a(\nabla u_0^T)\De$,  
$u_0^T:= e^{T\De}u_0$ and $e^{t\De}$ denotes the semigroup generated by $\De$ on $\mathbb{T}$. 
The term $\zeta =\zeta(u, u')\in C_TC^{\a+ \b -2}$ is
defined in \eqref{eq-2.8-0307} 
 below, which is considered as a remainder term in our analysis.  
Moreover, we assume $v(0) =u_0$. Note that $u'=v'$ is equivalent to $u'=\frac{g(\nabla u)}{a(\nabla u)}$.
The sum of the second and  third terms in the right hand side of 
\eqref{eq-2.8-0323} will be denoted by $\e_1(u,u')$, and the fourth term by $\e_2(u, u')$, respectively, and estimated in Lemma  \ref{lem:5.1} below.

The key point in  our analysis is to rewrite the SPDE \eqref{eq:Q-ab}
as in the form
\begin{align}\label{eq-2.10-0326}
\mathcal{L}^0 u=\big( a(\nabla u) -a(\nabla u_0^T) \big) \De u + g(\nabla u) \cdot \xi.
\end{align}
Although the leading part in the above equation is still the term 
$\big( a(\nabla u) -a(\nabla u_0^T) \big) \De u$, we can show that 
it is well-defined when $u$ is paracontrolled by $X$.
To make it clear, we generalize Lemma 3 of \cite{BDH-19} to the next result, which
is important for our goal. 

\begin{lem} \label{lem-2.4-0319}
For any $f\in \mathcal{L}_T^\b$ and $u_0\in C^\alpha$,  the following intertwining continuity estimate holds.
\begin{align*}
\|\mathcal{L}^0\big(\bar{\Pi}_f X \big) -\Pi_{a(\nabla u_0^T)f}(-\De X)\|_{C_TC^{\a +\b -2}} \lesssim 
\big(1 + T^{-\frac{\gamma -\a}{2}}\|u_0\|_{C^\a}\big)
 \|f\|_{\mathcal{L}_T^\b}\|X\|_{C^\a}.
\end{align*}
\end{lem}
\begin{proof}
This estimate  can be shown analogously to Proposition 12 of \cite{BDH-19}.
One of the key points in the proof of  Proposition 12 of \cite{BDH-19}
is Lemma 11 of  \cite{BDH-19}. This lemma can be generalized to
our case, that is,   
for any $f\in C^{\b}, \ g\in C^{\ga -1}$ and $v\in C^{\a-2}$, we have
\begin{align*}
\|g\Pi_f v -\Pi_{gf}v\|_{C^{\a +\b -2} }\ \lesssim \|f\|_{C^\b} 
 \|g\|_{C^{\ga -1}}  \|v\|_{C^{\a-2}};
\end{align*} 
recall that $\a +\ga>3$.
On the other hand, by \eqref{eq-2.11-0424}  for $F=a$ in Lemma \ref{lem-2.8-0424} and the Schauder estimate on $u_0^T$, see \eqref{eq-3.1-430} in Lemma \ref{lem-2.3-0309} below,
we observe that 
\begin{align} \label{eq-2.7-0321}
\|a(\nabla u_0^T)\|_{C^{\gamma -1}} 
\lesssim  \|a\|_{C^1}(1+ \|u_0^T\|_{C^{\ga}}) 
\lesssim  \|a\|_{C^1} \big(1+ T^{-\frac{\ga -\a}{2}} \|u_0\|_{C^\a} \big).
\end{align}
We can now complete the proof of this lemma by mimicking the proof of Proposition 12 of \cite{BDH-19}.
\end{proof}
Note that by the definition of $X$, 
\begin{align}\label{eq-3.42-0420}
\|X\|_{C^\a}\lesssim \|\xi\|_{C^{\a-2}}
\end{align} holds,
see p.\ 40 of \cite{GIP-15} for example.   In particular, the estimate in Lemma
\ref{lem-2.4-0319} is given by $\|\xi\|_{C^{\a-2}}$ instead of $\|X\|_{C^\a}$,
see also Remark \ref{rem:3.1-A} below.

To derive the fixed point problem for the map $\Phi$, in \eqref{eq-2.10-0326}, we first expand
\begin{equation} \label{eq:2g}
g(\nabla u) \cdot \xi = \Pi_{g(\nabla u)} \xi + \Pi_\xi g(\nabla u) + \Pi(g(\nabla u), \xi ),
\end{equation}
where, as in p.\ 40 of \cite{GIP-15}, by Lemma 2.7 (Taylor expansion)
of \cite{GIP-15}, we rewrite the resonant term as follows:
\begin{align*}
\Pi(g(\nabla u), \xi ) = g'(\nabla u) \Pi(\nabla u, \xi )
+P_g(\nabla u,\xi),
\end{align*}
where $P_g$ was denoted by $\Pi_g$ in Lemma 2.7 of \cite{GIP-15}.  
Note that 
$P_g(\nabla u,\xi)  \in C_TC^{3\a-4}$ is a good term, where $3\a-4$ is obtained 
from $(\a-1)+(\a-1)+(\a-2)$.  However, since 
\begin{equation} \label{eq:5.6}
\nabla u= \nabla\bar\Pi_{u'} X +
\nabla u^\sharp =  \bar\Pi_{u'}\nabla X +
\nabla u^\sharp +R_1
\end{equation}
with $R_1 := [\nabla,\bar\Pi_{u'}]X \in C_TC^{\a +\b -1}$ from \eqref{eq:5.5} and the continuity result
\eqref{eq-4.1-0406} on the commutator $[\nabla,\bar\Pi]$ given in Lemma \ref{lem:37}, we have
\begin{align*}
\Pi(\nabla u,\xi)
& = \Pi(\bar\Pi_{u'}\nabla X,\xi)
 + \Pi(\nabla u^\sharp, \xi) +\Pi(R_1,\xi)\\
 & = u' \Pi(\nabla X,\xi) + \bar C(u', \nabla X,\xi) 
 + \Pi(\nabla u^\sharp, \xi) +\Pi(R_1,\xi),
 \end{align*}
 where $\bar{C}$ is the modified commutator defined in Lemma \ref{lem modified commutator}.
Note that although $\Pi(\nabla u^\sharp, \xi)$ has nice spatial regularity at each $t>0$, it  should be evaluated in some proper space
weighted in time  under some regularity assumption on $u^\sharp$ because of the explosive temporal property as $t\downarrow 0$, see Lemma \ref{lem:5.1} for details.

 By Lemma \ref{lem modified commutator},
 $\bar C(u', \nabla X, \xi)\in C_TC^{2\a +\b -3} $ is a nice function.  Therefore, showing that 
$\Pi(R_1,\xi) \in C_TC^{2\a +\b -3} $ is also nice by the regularity of $R_1$ and $2\a +\b -3>0$, we have
\begin{align*}
\Pi(g(\nabla u), \xi ) = g'(\nabla u) u' \Pi(\nabla X, \xi ) 
+g'(\nabla u) \Pi(\nabla u^\sharp, \xi ) + A_1,
\end{align*}
where
\begin{align}\label{eq:A1}
A_1 = P_g(\nabla u,\xi) + g'(\nabla u) \big\{\bar C(u', \nabla X,\xi) 
 +\Pi(R_1,\xi)\big\}  \in C_TC^{2\a +\b -3}.
\end{align} 
Hence, plugging the above equation into \eqref{eq:2g}, we have
\begin{align}\label{eq-2.13-0323}
g(\nabla u) \cdot \xi = \Pi_{g(\nabla u)} \xi + \Pi_\xi g(\nabla u)
+ g'(\nabla u) u' \Pi(\nabla X, \xi ) 
+g'(\nabla u) \Pi(\nabla u^\sharp, \xi ) + A_1.
\end{align} 

Next, for the first term in \eqref{eq-2.10-0326}, by \eqref{eq:5.5}, we have
\begin{align*} 
\big(a(\nabla u)- a(\nabla u_0^T) \big) \De u
= \big(a(\nabla u)- a(\nabla u_0^T)\big) 
\big( \De\bar\Pi_{u'}X+ \De u^\sharp\big).
\end{align*}
Then, by the continuity result \eqref{eq-4.2-0406} on the commutator $[\De,\bar\Pi]$
given in Lemma \ref{lem:37}, we have
\begin{align*}
\De\bar\Pi_{u'}X & = \bar\Pi_{u'} \De X +R_2 \\
& = -\bar\Pi_{u'} \xi +R_2
\end{align*}
with $R_2:= [\De,\bar\Pi_{u'}]X \in C_TC^{\a + \b -2}$
(note that $\bar\Pi_{u'} 1=0$ by definition), so that 
\begin{align*}
\big(a(\nabla u)- a(\nabla u_0^T) \big) \De u
= - \big(a(\nabla u)- a(\nabla u_0^T)\big) \bar\Pi_{u'}\xi  
+ \big(a(\nabla u)- a(\nabla u_0^T)\big)(\De u^\sharp+R_2),
\end{align*}
which, by  Lemma \ref{lem-2.4-0405},  is further rewritten as
\begin{align} \label{eq-2.11-0323}
\big(a(\nabla u)- a(\nabla u_0^T) \big) \De u
=& - \Pi_{(a(\nabla u)- a(\nabla u_0^T))u'}\xi 
- \Pi((a(\nabla u)- a(\nabla u_0^T)), \bar\Pi_{u'}\xi )  \\
& + \big(a(\nabla u)- a(\nabla u_0^T)\big) \De u^\sharp+ A_2, \notag
\end{align}
where
\begin{align}  \label{eq:A2}
A_2 = & -R(a(\nabla u)- a(\nabla u_0^T), u';\xi)  - \Pi_{\bar\Pi_{u'}\xi}
(a(\nabla u)- a(\nabla u_0^T))  \\
& \hskip 20mm  
+ (a(\nabla u)- a(\nabla u_0^T)) R_2
\in C_TC^{\a +\b -2}; \notag
\end{align}
see Lemma \ref{lem-2.4-0405} for the definition of the first term $R$ 
in $A_2$.

As we stated above, the term including $u^\sharp$ should be estimated in a weighted space with weight in time under the regularity assumption on  $u^\sharp$. Hence we can not consider it as a remainder such as $A_2$.

To derive the map $\Phi$, we first note that the resonant term
 $\Pi(a(\nabla u_0^T), \bar\Pi_{u'}\xi )$ in \eqref{eq-2.11-0323} is well-defined for $u_0 \in C^{\a}$. In fact,   by
 Lemma \ref{lem-5.1-0301}-{\rm (ii)}, \eqref{eq-2.7-0321} and  Lemma \ref{lem:bar}-{\rm (i)},  we have that 
\begin{align}\label{eq:2.18}
& \|\Pi(a(\nabla u_0^T), \bar\Pi_{u'}\xi )\|_{C_TC^{\a + \gamma -3}} \\ \lesssim & \|a(\nabla u_0^T)\|_{C^{\gamma-1}}   \|\bar\Pi_{u'}\xi  \|_{C_TC^{\alpha-2}}  \notag \\ 
\lesssim  & \|a\|_{C^1}(1+ T^{-\frac{\gamma-\a}{2}}\|  u_0\|_{C^\a} )\|u'\|_{\mathcal{L}_T^\b}\|\xi\|_{C^{\alpha-2}};
\notag
\end{align}
recall that  for any $\a + \gamma -3>0$. 
For the resonant term $\Pi(a(\nabla u), \bar\Pi_{u'}\xi )$  in \eqref{eq-2.11-0323}, 
by Lemma 2.7 of \cite{GIP-15} (Taylor expansion) and \eqref{eq:5.6},
\begin{align*}
\Pi(a(\nabla u), \bar\Pi_{u'}\xi )
& = a'(\nabla u) \Pi(\nabla u, \bar\Pi_{u'} \xi) + P_a(\nabla u, \bar\Pi_{u'}\xi)\\ 
& = a'(\nabla u) \Pi(\bar{\Pi}_{u'}\nabla X+
\nabla u^\sharp+R_1, \bar\Pi_{u'}\xi) + P_a(\nabla u, \bar\Pi_{u'}\xi).
\end{align*}
Lemma \ref{lem modified commutator} yields
\begin{align*}
\Pi(\bar{\Pi}_{u'}\nabla X,\bar\Pi_{u'}\xi)
&=u'\Pi(\nabla X,\bar\Pi_{u'}\xi)+\bar{C}(u',\nabla X,\bar\Pi_{u'}\xi)\\
&=(u')^2\Pi(\nabla X,\xi) 
+u'\bar{C}(u',\xi,\nabla X) 
+\bar{C}(u',\nabla X,\bar\Pi_{u'}\xi), 
\end{align*}
where the last two commutators are in $C_TC^{2\alpha+\beta-3}$.
Thus, 
\begin{align*}
\Pi(a(\nabla u), \bar\Pi_{u'}\xi )
= a'(\nabla u) (u')^2 \Pi(\nabla X, \xi ) 
+a'(\nabla u)\Pi(\nabla u^\sharp,\bar\Pi_{u'}\xi)
+ A_3,
\end{align*}
where
\begin{align}  \label{eq:A3}
A_3
= P_a(\nabla u, \bar\Pi_{u'}\xi) +  a'(\nabla u) \big\{
u'\bar{C}(u',\xi,\nabla X) + \Pi(R_1, \bar\Pi_{u'}\xi) + \bar{C}(u',\nabla X,\bar\Pi_{u'}\xi) 
 \big\}
\end{align}
belongs to $C_TC^{2\alpha+\beta-3}$, since $\bar\Pi_{u'}\xi\in C_TC^{\alpha-2}$ by Lemma \ref{lem:bar}-{\rm(i)},

Thus, noting \eqref{eq-2.11-0323}, we have
\begin{align}\label{eq-2.16-0323}
& \big(a(\nabla u)- a(\nabla u_0^T) \big) \De u \\
=&  - \Pi_{(a(\nabla u)- a(\nabla u_0^T))u'}\xi 
+ \Pi( a(\nabla u_0^T), \bar\Pi_{u'}\xi )  + \big(a(\nabla u)- a(\nabla u_0^T)\big) \De u^\sharp+ A_2, \notag \\
& -\big\{ a'(\nabla u) (u')^2 \Pi(\nabla X, \xi ) 
+a'(\nabla u)\Pi(\nabla u^\sharp,\bar\Pi_{u'}\xi)
+ A_3 \big\}.  \notag 
\end{align}

Thus, assuming that $\Pi(\nabla X,\xi)$ is nicely defined
 (see Lemma \ref{lem:2.9} below) and using \eqref{eq-2.13-0323}, \eqref{eq-2.16-0323},
we can rewrite the equation \eqref{eq-2.10-0326} and therefore \eqref{eq:Q-ab} as
\begin{align} \label{eq-2.17-0323}
\mathcal{L}^0u = & \Pi_{g(\nabla u)- (a(\nabla u)- a(\nabla u_0^T))u'}\xi
+g'(\nabla u)\Pi(\nabla u^\sharp,\xi)-a'(\nabla u)\Pi(\nabla u^\sharp,\bar\Pi_{u'}\xi)
\\& +\big(a(\nabla u)- a(\nabla u_0^T)\big) \De u^\sharp  +  \zeta,
\notag
\end{align}
where
\begin{align} \label{eq-2.8-0307}
 \zeta=\zeta(u, u')  = & \Pi_\xi g(\nabla u) +g'(\nabla u) u' \Pi(\nabla X, \xi ) 
 + A_1  + \Pi(a(\nabla u_0^T), \bar\Pi_{u'}\xi ) +A_2\\
& -\big\{ a'(\nabla u) (u')^2 \Pi(\nabla X, \xi ) + A_3 \big\}.  \notag
\end{align}
We have $\zeta\in C_TC^{\a+ \b-2}$.
This leads to \eqref{eq-2.8-0323}  with \eqref{eq-2.7-0323},
more precisely, \eqref{eq-2.17-0323} coincides with \eqref{eq-2.8-0323} if $u=v$ and $u'=v'$.

By the above analysis, it is clear that the remainder $\zeta$ 
is described explicitly in terms of multilinear maps of $\nabla u, u'$ 
or $C_b^2$ functions of $\nabla u$, whose  $C^{\a +\b-2}$-norm 
depends polynomially on the norm $\|(u, u')\|_{\a, \b, \gamma}$ and the data. 
Hereafter,  the data mean  $\|u_0\|_{C^\a}$, $\|X\|_{C^\a}$, $\|\xi\|_{C^{\a -2}}$, $\|\Pi(\nabla X, \xi)\|_{C^{2\a-3}}$ and  the norms relative to the coefficients $a$ and $g$.
Moreover, as a function defined on ${\bf C}_{\a, \b, \gamma}(X)$, 
we will show that $\zeta$ is locally Lipschitz continuous with a Lipschitz constant
depending polynomially on $\|(u, u')\|_{\a, \b, \gamma}$ and the data,
see Proposition \ref{lem:Lip-zeta} in Section \ref{sec-2.2-0307} 
for details. 

Throughout this paper, we  denote by $K_0$ a generic constant
 depending possibly on $\|u_0\|_{C^\a}$ and the norms relative to  
$a$ and $g$, but not on $X,$ $\xi$ and $\Pi(\nabla X, \xi)$. 
In addition, for simplicity,  we denote  by $K(\|(u,u')\|_{\a, \b, \gamma})$ 
a positive constant depending polynomially  on both  $\|(u,u')\|_{\a, \b, \gamma}$ and 
$K_0$ of the above type, and by  $K(\|(u_1,u_1')\|_{\a, \b, \gamma},$ 
$\|(u_2,u_2')\|_{\a, \b, \gamma})$ a positive constant depending  
polynomially on $\|(u_1,u_1')\|_{\a, \b, \gamma}$,  
$\|(u_2,u_2')\|_{\a, \b, \gamma}$ and $K_0$ when we study the Lipschitz property.
All of the constants may change from line to line.

\section{Solving fixed point problem} \label{sec-2.2-0307}

In this section, we solve the fixed point problem for the map $\Phi$
on ${\bf C}_{\a,\b,\ga}(X)$ defined by \eqref{eq-2.4-0304} and give
the proof of Theorem \ref{thm:main}.

\subsection{Formulation of the result and proof of Theorem \ref{thm:main}}

To solve our equation \eqref{eq:Q-ab} by the fixed point theorem, 
for $\la>0$, we set
$$
\mathcal{B}_T(\la) := \Big\{ (u,u')\in {\bf C}_{\a,\b,\ga}(X);\ 
u(0)=u_0,\ u'(0)= \frac{g(\nabla u_0)}{a(\nabla u_0)},\ 
\|(u,u')\|_{\a,\b,\ga} \le \la \Big\},
$$
where $u_0\in C^\a, \a\in (\frac43,\frac32)$, is given as in Theorem \ref{thm:main}.
Observe that the initial data $u(0)$ and $u'(0)$ are fixed in $\mathcal{B}_T(\la)$.

The main theorem of this section is the following.
\begin{thm} \label{thm-2.1-0304}
{\rm (i)} There exist a large enough $\lambda>0$ and a small enough $T>0$ such that
the map $\Phi$ defined by \eqref{eq-2.4-0304} is contractive
from $\mathcal{B}_T(\la)$ into itself.  In particular, $\Phi$ has a unique fixed point
on $[0,T]$ for $T>0$ small enough, which solves 
the paracontrolled SPDE \eqref{eq:Q-ab} locally in time. \\
{\rm (ii)} The map $\Phi$ depends continuously on the enhanced noise 
$\hat{\xi}:=(\xi, \Pi(\nabla X, \xi))\in C^{\a-2}\times C^{2\a -3}$ and its contractivity
on  $\mathcal{B}_T(\la)$ is  locally uniform in $\hat{\xi}$.  In particular, the unique fixed point 
of $\Phi$ in  $\mathcal{B}_T(\la)$ inherits the continuity in $\hat{\xi}$. 
\end{thm}

Once  Theorem \ref{thm-2.1-0304} is shown, one can prove Theorem \ref{thm:main}.

\begin{proof}[Proof of Theorem \ref{thm:main}]
Let the initial value $u_0\in C^\a$ of  \eqref{eq:Q-ab} be given.  The fixed point
$(u,u')$ of the map $\Phi$ satisfies \eqref{eq-2.7-0323} and \eqref{eq-2.8-0323}
with $v=u$ and $v'=u'$ so that, by the arguments in Section \ref{sec-2.1-0307},
we see that $u$ is a solution of SPDE \eqref{eq:Q-ab}  paracontrolled by $X$.
Conversely, if $u$ is a solution of \eqref{eq:Q-ab} paracontrolled by $X$,
it is a fixed point of $\Phi$.  Therefore, Theorem \ref{thm-2.1-0304}-(i)
shows the existence and uniqueness of a local solution of the SPDE \eqref{eq:Q-ab}.

For each enhanced noise $\hat{\xi}=(\xi, \Pi(\nabla X, \xi))\in C^{\a-2}\times C^{2\a -3}$,
let us denote the fixed point by $\mathcal{J}(\hat{\xi})$. 
Given a spatially mild noise $\eta$  such as the smeared noise $\xi^\e$
in Theorem \ref{thm:main} or Lemma \ref{lem:2.9} on $\mathbb{T}$, let us denote
 by $I(\eta)$ the unique solution 
of  the well-posed equation \eqref{eq:Q-ab} with the mild noise $\eta$ instead of $\xi$.
We have that $\mathcal{J}(\hat{\eta}):=\mathcal{J}(\eta, \Pi(\nabla Y, \eta))$ 
extends the map $I(\eta)$ in
the sense that $\mathcal{J}(\hat{\eta})=I(\eta)$ for any mild noise $\eta$, where 
$Y=(-\De)^{-1}Q\eta$. In particular, the convergence result in
Theorem \ref{thm:main} is obtained by the continuity of the map $\mathcal{J}$ 
shown in Theorem \ref{thm-2.1-0304}-(ii) and Lemma \ref{lem:2.9} in Section
\ref{sec:Renorm}.
\end{proof}

\subsection{Lipschitz estimates and Schauder estimates}

We will give the proof of Theorem \ref{thm-2.1-0304} at the end of this section.
Here, we prepare growth and Lipschitz estimates (Lemmas \ref{lem:5.1} and \ref{lem:5.2}
and Proposition \ref{lem:Lip-zeta}) and Schauder estimates (Lemmas \ref{lem-2.3-0309}
and \ref{lem:5.3} and Corollary \ref{cor:5.4}).

The next results are taken from A.7 and A.8 of \cite{GIP-15},   Corollary 2.7 of \cite{Ho} and  Lemma 6 of \cite{BDH-19}. Let us denote by $P_t=e^{t\De}, t\ge 0$ the semigroup generated 
by $\De$ on $\mathbb{T}$. 
\begin{lem} (Schauder estimates) \label{lem-2.3-0309}
Let  $\alpha \in \R$ and $t\in(0, T]$. Then the following hold.\\
{\rm (i)} (Lemma A.7 \cite{GIP-15} and  Lemma 6  \cite{BDH-19})  For $\beta \geq 0$ and $u\in C^\a$, 
\begin{align}\label{eq-3.1-430}
\|P_t u\|_{C^{\a + \b}} \lesssim_T t^{-\frac{\beta}{2}}\|u\|_{C^a},  \ t\in (0,T].
\end{align}
Here and in the sequel, ``$\lesssim_T$'' means the implicit multiplicative
constant in the right hand depends uniformly on $T$ for all $t\in (0,T]$.
For  $\beta \geq 0$,
\begin{align*}
\|P_t u\|_{C^\b} \lesssim_T t^{-\frac{\beta}{2}}\|u\|_{L^\infty},  \ t\in (0,T]
\end{align*}
and  conversely 
for   $\b<0$
\begin{align} \label{eq-3.3-430}
\|P_t u\|_{L^\infty} \lesssim_T t^{\frac{\b}{2}}\|u\|_{C^\b},  \ t\in (0,T].\end{align}
\noindent
{\rm (ii)}  (Corollary 2.7 \cite{Ho} and Lemma A.8 \cite{GIP-15})
Let  {\rm Id} denote the identity operator and $\beta \in [0,2]$. 
Then,  we have for all $t\geq 0$
\begin{align}
\|(P_t -{\rm Id}) u\|_{C^\alpha} \lesssim & t^{\frac{\beta}{2} }\|u \|_{C^{\alpha +\beta}}, \label{eq-3.4-430}
\\
\|(P_t -{\rm Id}) u\|_{L^\infty} \lesssim & t^{\frac{\beta}{2} }\|u \|_{C^\b}.                 \label{eq-3.5-430}
\end{align} 
{\rm (iii)}  (Lemma A.9 \cite{GIP-15})  For any $u\in C_TC^\a$, let us denote by $U(t)$ the convolution of $P_t$ and $u=u(\cdot)$, that is, 
$
U(t)=\int_0^t P_{t-s}u(s)ds.
$
\vspace{-3mm}
\begin{itemize}
 \setlength\itemsep{0em}
\item For all $\beta \in [0,1)$, we have 
$
t^\b\|U(t)\|_{C^{\a +2}} \lesssim  \sup_{s\in [0,t]}\big( s^\b\|u(s)\|_{C^{\a}}\big),\ t\in [0,T]. 
$
\item If $\a \in (-2, 0)$, then
$
\|U\|_{C_T^{\frac{\a +2}{2} }L^\infty}\lesssim \|u\|_{C_TC^\a}.
$
\end{itemize}
{\rm(iv)} (Lemma 6 \cite{BDH-19})  Let us denote by $Q_t$ the semigroup generated  by 
$\nabla{(b\nabla \cdot )}$ for some positive function $b\in C_b^2(\mathbb{T})$. Then 
the above statements in {\rm(i)}-{\rm(iii)} hold for $Q_t$ instead of $P_t$ with implicit multiplicative constants depending only on the $C^\a$-norm of $b$ and positive lower bound of $b$.
\end{lem}

As we stated in Section \ref{sec-2.1-0307}, one task is to evaluate the terms including  $u^\sharp$. It is formulated in the following lemma.
\begin{lem}  \label{lem:5.1}
For ${\bf u} := (u,u') \in \mathcal{B}_T(\la)$, set
\begin{align*}
& \e_1({\bf u}) := \e_1(u,u') = 
g'(\nabla u)\Pi(\nabla u^\sharp,\xi)-a'(\nabla u)\Pi(\nabla u^\sharp,\bar\Pi_{u'}\xi),\\
& \e_2({\bf u}) := \e_2(u,u') = \big( a(\nabla u) - a(\nabla u_0^T) \big)
\De u^\sharp.
\end{align*}
Then we have the local growth estimates
\begin{align}
& \sup_{0<t\le T} t^{\frac{\ga-\a}2} \|\e_1({\bf u})(t)\|_{C^{2\a -3}}\lesssim K(\|{\bf u}\|_{\a,\b,\ga}) 
(1+\|\xi\|_{C^{\a -2}})\|\xi\|_{C^{\a -2}},    \label{eq-3.1-429}\\
& \sup_{0<t\le T} t^{\frac{\ga-\a}2} \|\e_2({\bf u})(t)\|_{C^{\ga-2}}
\lesssim  T^{\frac{\a -\b -1}{2}} K(\|{\bf u}\|_{\a,\b,\ga}) 
 (1+\|\xi\|_{C^{\a -2}})+K_0  \label{eq-3.2-429}
\end{align}
for two constants $K(\|{\bf u}\|_{\a,\b,\ga})$ and $K_0$.

Moreover, we have  the local Lipschitz estimates: for any ${\bf u}_1 := (u_1,u_1')$ and ${\bf u}_2 := (u_2,u_2') \in \mathcal{B}_T(\la)$,
\begin{align}\label{eq-5.7-0302}
& \sup_{0<t\le T} t^{\frac{\ga-\a}2} 
\|\big(\e_1({\bf u}_1) -\e_1({\bf u}_2 )\big)(t) \|_{C^{2\a -3} 
} \\
& \quad \lesssim K(\|{\bf u}_1 \|_{\a,\b,\ga}, \|{\bf u}_2\|_{\a,\b,\ga})
\| {\bf u}_1 - {\bf u}_2\|_{\a, \b, \gamma}
(1+\|\xi\|_{C^{\a -2}})^2 \|\xi\|_{C^{\a -2}},  \notag \\
& \sup_{0<t\le T} t^{\frac{\ga-\a}2}   
\|\big(\e_2({\bf u}_1) -\e_2({\bf u}_2) \big) (t) \|_{C^{\ga-2}} \label{eq-5.8-0302} \\ 
& \quad \lesssim T^{\frac{\a -\b -1}{2}}K(\|{\bf u}_1\|_{\a,\b,\ga}, 
\|{\bf u}_2\|_{\a,\b,\ga}) \| {\bf u}_1 - {\bf u}_2\|_{\a, \b, \gamma}
(1+\|\xi\|_{C^{\a -2}})^2
 \notag
\end{align}
hold for some constant  $K(\|{\bf u}_1\|_{\a,\b,\ga}, 
\|{\bf u}_2\|_{\a,\b,\ga})$.
\end{lem}

\begin{rem} \label{rem:3.1-A}
In the following in the proofs of Lemmas \ref{lem:5.1},  \ref{lem:5.2},
Corollary \ref{cor:5.4}, Proposition \ref{lem:Lip-zeta} and Lemmas \ref{lem-A0}-\ref{lem-A3},
we keep the norm $\|X\|_{C^\a}$ in estimates to make its origin clear, but eventually
by \eqref{eq-3.42-0420}, we bound it by $\|\xi\|_{C^{\a -2}}$.
\end{rem}

\begin{proof}[Proof of Lemma \ref{lem:5.1}]
To prove this lemma, let us  first note that for any ${\bf u}\in \mathcal{B}_T(\la)$, 
we have
\begin{align}\label{eq-3.5-0321}
\|u\|_{\mathcal{L}_T^{\a}} \lesssim \|u'\|_{\mathcal{L}_T^\b} 
\|X\|_{C^\a}
 +\|u^\sharp \|_{\mathcal{L}_T^\a} \leq (1+ \|X\|_{C^\a}) 
\|\bf{u}\|_{\a, \b, \ga},
\end{align}
which can be shown by 
Lemma 13 of \cite{BDH-19}, see also Lemma 2.10 of \cite{GP-17}. 
Indeed, the proof of  Lemma 13 of \cite{BDH-19} shows  
that $
\|\bar{\Pi}_fg \|_{\mathcal{L}_T^{\a}} \lesssim \|f\|_{\mathcal{L}_T^\b}\|g\|_{C^\a}
$
for any $ f\in \mathcal{L}_T^\b, \b\in (0,1)$ 
and $g\in C^\a,  \a\in (0,2)$. So we  obtain \eqref{eq-3.5-0321} by recalling \eqref{eq:5.5}.
By Lemma \ref{lem:bar}-{\rm (i)}, we easily have the weak estimate $\|u\|_{C_TC^{\a}} \lesssim  (1+ \|X\|_{C^\a}) \|{\bf u}\|_{\a, \b, \ga}$, which is enough for us
to prove the estimates on  $\e_1({\bf u})$. But, for $\e_2({\bf u})$ and Lemma \ref{lem:5.2} below, we have to use \eqref{eq-3.5-0321}. So we state it here.

We now prove \eqref{eq-3.1-429}.
Noting that $\alpha<\gamma$ and $0<\alpha+\gamma-3<\beta$, we have
\begin{align} \label{eq-3.4-0429}
&\|\e_1({\bf u})(t)\|_{C^{2\a-3}} 
\lesssim \|\e_1({\bf u})(t)\|_{C^{\alpha+\gamma-3}} \\
\lesssim &\|g'(\nabla u)(t)\|_{C^{\b}}
\| \Pi(\nabla u^\sharp,\xi)(t) \|_{C^{\a+\ga-3}}
+\|a'(\nabla u)(t)\|_{C^\beta}\|\Pi(\nabla u^\sharp,\bar\Pi_{u'}\xi)(t)\|_{C^{\alpha+\gamma-3}}   \notag\\
\lesssim &\|g'(\nabla u)(t)\|_{C^{\b}}
\| \nabla u^\sharp(t) \|_{C^{\ga-1}} \|\xi\|_{C^{\alpha-2}} \notag\\
&+\|a'(\nabla u)(t)\|_{C^\beta}
\|\nabla u^\sharp(t)\|_{C^{\gamma-1}}\|u'\|_{C_TC^\beta}\|\xi\|_{C^{\alpha-2}}
 \notag
\\ 
\lesssim & 
(\|g'(\nabla u)(t)\|_{C^{\b}} +\|a'(\nabla u)(t)\|_{C^\beta})
(1+\|u'\|_{C_TC^\beta})
\|  u^\sharp(t) \|_{C^{\ga}} \|\xi\|_{C^{\alpha-2}}, \notag
\end{align} 
where we have used Lemma \ref{lem-5.1-0301}-{\rm (i)} and {\rm (ii)} 
for the first and second  inequalities, respectively, and  
Lemma \ref{lem-5.1-0301}-{\rm (ii)} and Lemma
\ref{lem:bar} for the third inequality.

Using \eqref{eq-2.11-0424} 
in Lemma \ref{lem-2.8-0424} together with 
\eqref{eq-3.5-0321}, we have
\begin{align}\label{eq-3.4-0424}
\|g'(\nabla u)(t)\|_{C^\b} +\|a'(\nabla u)(t)\|_{C^\b} 
\lesssim  & (\| g'\|_{C^1}+\| a'\|_{C^1}) (1+\|u(t)\|_{C^{\b+1}}) 
 \\
\lesssim &   (\| g'\|_{C^1}+\| a'\|_{C^1}) \big(1+ (1+ \|X\|_{C^\a}) 
\|{\bf u}\|_{\a, \b, \ga}  \big)   \notag \\
\lesssim &   (\| g'\|_{C^1}+\| a'\|_{C^1}) (1+  
\|{\bf u}\|_{\a, \b, \ga}) (1+ \|X\|_{C^\a}).  \notag
\end{align}
Combining \eqref{eq-3.4-0429} with \eqref{eq-3.4-0424}, we have
\begin{align} \label{eq-3.6-0321}
&\|\e_1({\bf u})(t)\|_{C^{2\a-3}} 
\\ 
\notag
\lesssim&\big(\|g'\|_{C^1}+\|a'\|_{C^1})(1+\|{\bf u}\|_{\alpha,\beta,\gamma})^2 (1+\|X\|_{C^\alpha})
\|u^\sharp(t)\|_{C^\gamma}\|\xi\|_{C^{\alpha-2}},
\end{align}
and then by multiplying both sides of \eqref{eq-3.6-0321} 
by $t^{\frac{\ga-\a}2}$,
we obtain  \eqref{eq-3.1-429} with the constant
$$
K(\|{\bf u}\|_{\a,\b,\ga})
=\big(\|g'\|_{C^1}+\|a'\|_{C^1})(1+\|{\bf u}\|_{\alpha,\beta,\gamma})^2
\|{\bf u}\|_{\alpha,\beta,\gamma}.
$$
We remark all of the other constants $K$ and $K_0$ used in this paper can be obtained by the similar way. So the details for them will be omitted in the sequel.

The local Lipschitz estimate \eqref{eq-5.7-0302} of $\e_1$ is 
proved as follows. 
By the multilinearity of the resonant and the modified paraproduct,  
we have
\begin{align}\label{eq-5.11-0302}
& \|\big(\e_1({\bf u}_1) -\e_1({\bf u}_2 )\big)(t) \|_{C^{2\a -3}}  \\
\leq &  
\| \big( g'(\nabla u_1)   \Pi(\nabla ( u_1^\sharp - u_2^\sharp),\xi)  - a'(\nabla u_1) \Pi(\nabla( u_1^\sharp - u_2^\sharp),\bar{\Pi}_{u_1'}\xi) \big)(t) \|_{C^{2\a -3}}
 \notag \\
& + \|\big(\big( g'(\nabla u_1) -g'(\nabla u_2) \big)  \Pi(\nabla u_2^\sharp,\xi) -  \big( a'(\nabla u_1) -a'(\nabla u_2) \big)  
\Pi(\nabla u_2^\sharp,\bar{\Pi}_{u_1'}\xi)
\big) (t)\|_{C^{2\a -3}}\notag \\
& + \| \big( a'(\nabla u_2) \Pi(\nabla u_2^\sharp,\bar{\Pi}_{u_1'-u_2'}\xi)
\big) (t)\|_{C^{2\a -3}}\notag \\
=:&  I(t) +I\!I(t) + I\!I\!I(t). \notag
\end{align}
From \eqref{eq-3.6-0321}, it follows that 
\begin{align}\label{eq-5.12-0302}
I(t)\lesssim K(\|{\bf u}_1\|_{\a, \b, \ga}) \|\big(u_1^\sharp- u_2^\sharp\big)(t)\|_{C^{\ga}} (1+\|X\|_{C^\a})\|\xi\|_{C^{\a -2}}, \ t\in [0,T].
\end{align}
To evaluate the second term $I\!I(t)$, we first note that 
\begin{align}\label{eq-3.8-03038}
& \|\big(g'( \nabla u_1) -g'(\nabla u_2) \big)(t)\|_{C^\beta} +
 \|\big(a'( \nabla u_1) -a'(\nabla u_2) \big)(t)\|_{C^\beta} \\
\lesssim & (\|g'\|_{C^2} +\|a'\|_{C^2} )
\big(1+ (1+\|X\|_{C^\a})\|{\bf u}_1 \|_{\a, \b, \ga} \big)
(1+\|X\|_{C^\a}) \|{\bf u}_1- {\bf u}_2 \|_{\a, \b, \ga}
 \notag 
\\
\lesssim &
K(\|{\bf u}_1\|_{\a, \b, \ga})\|{\bf u}_1- {\bf u}_2\|_{\a, \b, \ga}
(1+\|X\|_{C^\a})^2, \  t\in [0,T], \notag
\end{align}
which follows from \eqref{eq-2.12-0424} and the similar arguments 
for \eqref{eq-3.4-0424}. 
We remark that \eqref{eq-3.4-0424} and \eqref{eq-3.8-03038}
also hold for $g, a$ instead of $g', a'$, which will be used in the
sequel.
Using this estimate and repeating essentially 
the same arguments for \eqref{eq-3.6-0321},  we have
\begin{align} \label{eq-3.16-503}
I\!I(t) 
\lesssim & K(\|{\bf u}_1\|_{\a, \b, \ga}, \|{\bf u}_2\|_{\a, \b, \ga})
\|{\bf u}_1-{\bf u}_2\|_{\a, \b, \ga}
\| u_2^\sharp(t)\|_{C^{\ga}} (1+\|X\|_{C^\a})^2 \|\xi\|_{C^{\a-2}} 
\end{align}
for all $t\in [0,T]$. 
Moreover, thanks to  \eqref{eq-3.4-0424}, 
the similar arguments give that 
\begin{align*}
I\!I\!I(t) \lesssim K(\|{ \bf u}_2\|_{\a, \b, \ga})\|{\bf u}_1-{\bf u}_2\|_{\a, \b, \ga}
\| u_2^\sharp(t)\|_{C^{\ga}} (1+\|X\|_{C^\a}) \|\xi\|_{C^{\a-2}}, \ t\in [0,T]. 
\end{align*}
Plugging this, \eqref{eq-5.12-0302} and \eqref{eq-3.16-503} into \eqref{eq-5.11-0302}, 
we  obtain \eqref{eq-5.7-0302}.

Let us now turn to the proofs of \eqref{eq-3.2-429} and \eqref{eq-5.8-0302}.
Since  $\beta >0> \gamma- 2$ and $\b + \ga -2>3\b-1>0$, Lemma \ref{lem-5.1-0301}-{\rm (ii)} implies for all $t\in [0,T]$,
\begin{align*}
\|\e_2({\bf u})(t)\|_{C^{\ga -2}}
\lesssim &
\|\e_2({\bf u})(t)\|_{C^{\b+ \ga -2}} \\
\lesssim & 
\|\big( a(\nabla u) - a(\nabla u_0^T) \big)(t) \|_{C^\b}
\|\De u^\sharp(t) \|_{C^{\gamma -2}}\\
\lesssim &
\| a(\nabla u) - a(\nabla u_0^T) \|_{C_TC^\b}
\| u^\sharp(t) \|_{C^{\gamma }}. \notag
\end{align*}
Then, we can obtain \eqref{eq-3.2-429} if there
exists a constant $K(\|\bf{u}\|_{\a, \b, \ga})$ such that 
\begin{align*}
 \| a(\nabla u) - a(\nabla u_0^T) \|_{C_TC^\b}\lesssim 
 T^{\frac{\a -\b -1}{2}}K(\|{\bf u}\|_{\a, \b, \ga}) (1+\|X\|_{C^\a})
+K_0.
\end{align*}
To obtain the important factor $T^{\frac{\a -\b -1}{2}}$,
we use Lemma \ref{lem-2.8-0424}-{\rm (ii)} and 
show  a stronger result: 
for any $u\in \mathcal{L}_T^\a$,  
\begin{align}\label{eq-5.8-0301}
\| a(\nabla u) - a(\nabla u_0^T) \|_{\mathcal{L}_T^\b} 
\lesssim T^{\frac{\a -\b -1}{2}}K(\|{\bf u}\|_{\a, \b, \ga}) (1+\|X\|_{C^\a})
+K_0,  
\end{align}
which will be frequently used in the sequel.

From \eqref{eq-2.14-0424} and 
$\beta < \alpha-1$, it follows that 
\begin{align}\label{eq-5.9-0301}
&  \| a(\nabla u) - a(\nabla u_0) \|_{\mathcal{L}_T^\b} \\
\lesssim & T^{\frac{\a -\b -1}{2}}\|a\|_{C^2} 
(1+ \| u_0\|_{C^{\alpha}})
\|u -u_0\|_{\mathcal{L}_T^{\alpha}}  \notag \\
\lesssim &T^{\frac{\a -\b -1}{2}}\|a\|_{C^2} 
(1+ \| u_0\|_{C^{\alpha}}) (\|{\bf u}\|_{\a, \b, \ga} + \|u_0\|_{C^\a}) (1+\|X\|_{C^\a}), \notag
\end{align}
where the strong estimate \eqref{eq-3.5-0321} has been used for 
the last inequality.
On the other hand,  applying \eqref{eq-2.11-0424} and then 
using \eqref{eq-3.4-430} in Lemma \ref{lem-2.3-0309}, we have 
\begin{align*}
\| a(\nabla u_0^T) - a( \nabla u_0)\|_{{C}^\b}
\lesssim \|a\|_{C^2}(1+\| u_0\|_{C^{\a}})
\| u_0^T- u_0\|_{C^{\a}} 
\lesssim  \|a\|_{C^2}(1+\|u_0\|_{C^\a})\|u_0\|_{C^\a}.
\end{align*}
Combining this with \eqref{eq-5.9-0301}, we obtain 
\eqref{eq-5.8-0301} and therefore complete the proof of \eqref{eq-3.2-429}.

Finally, let us show the local Lipschitz estimate \eqref{eq-5.8-0302} of $\e_2$.  
By repeating essentially the same arguments for \eqref{eq-5.9-0301}, we 
have 
\begin{align}\label{eq-3.13-0409}
 & \| a(\nabla u_1) - a(\nabla u_2) \|_{\mathcal{L}_T^\b}\\ 
\lesssim  &  T^{\frac{\a -\b -1}{2}}\|a\|_{C^2} 
(1+ \|{\bf u}_1 \|_{\a, \b, \ga})\|{\bf u}_1- {\bf u}_2 \|_{\a, \b, \ga}
(1+\|X\|_{C^\a})^2. \notag
\end{align}
The definition of $\e_2$  gives that
\begin{align*}
& \|\big(\e_2({\bf u}_1) -\e_2({\bf u}_2) \big) (t) \|_{C^{\ga-2}} \\
\leq & \|\big( \big( a(\nabla u_1) - a(\nabla u_2) \big) \De u_1^\sharp \big) (t)\|_{_{C^{\ga-2}}}
 +\|\big( \big( a(\nabla u_2) - a(\nabla u_0^T) \big) 
\De\big( u_1^\sharp -u_2^\sharp \big)\big)(t) \|_{_{C^{\ga-2}}}
\notag \\
\lesssim  & \|a(\nabla u_1) - a(\nabla u_2)\|_{\mathcal{L}_T^\beta}
\| u_1^\sharp(t)\|_{C^{\ga}}  
+\|  a(\nabla u_2) - a(\nabla u_0^T) \|_{\mathcal{L}_T^\b} 
 \|\big(u_1^\sharp -u_2^\sharp \big)(t) \|_{C^{\ga}} \notag
 \\
\lesssim  & T^{\frac{\alpha-\beta -1}{2}} K(\|{\bf u}_1\|_{\a, \b, \ga},  \|{\bf u}_2\|_{\a, \b, \ga}) (1+\|X\|_{C^\a})^2\\
& \times
\big(
\|{\bf u}_1 - {\bf u}_2 \|_{\a,\b, \ga}
\| u_1^\sharp(t)\|_{C^{\ga}}   +
 \|\big(u_1^\sharp -u_2^\sharp \big)(t) \|_{C^{\ga}} \big), \notag
\end{align*}
where  we have used \eqref{eq-3.13-0409} and \eqref{eq-5.8-0301}
with $u=u_2$ for the last inequality. Hence, we  complete the proof of \eqref{eq-5.8-0302} and then the proof of this lemma.
\end{proof}

The next lemma gives the local growth and local Lipschitz properties of the map $\Phi(u,u')$ in $u'$.
\begin{lem}  \label{lem:5.2}
For ${\bf u}_1 = (u_1,u_1')$ and  ${\bf u}_2 = (u_2,u_2')
\in \mathcal{B}_T(\la)$, set ${\bf v}_1 := \Phi({\bf u}_1 )= (v_1,v_1')$
and ${\bf v}_2 := \Phi({\bf u}_2)= (v_2,v_2')$. Then, we have the local growth estimate
\begin{align}\label{eq-2.13-0302}
& \| v_1'\|_{\mathcal{L}_T^\b} \lesssim 
 T^{\frac{\a-\b-1}2} K(\|{\bf u}_1\|_{\a, \b, \gamma})
 (1+\|\xi\|_{C^{\a -2}})  +K_0.
\end{align} 
Moreover, the following local Lipschitz estimate also holds. 
\begin{align}\label{eq-2.18-0304}
\|v_1'-v_2'\|_{\mathcal{L}_T^\b} \lesssim T^{\frac{\a-\b-1}2} 
K(\| {\bf u}_1\|_{\a,\b,\ga},  \| {\bf u}_2\|_{\a,\b,\ga})
\| {\bf u}_1 - {\bf u}_2\|_{\a,\b,\ga}(1+\|\xi\|_{C^{\a -2}})^2.
\end{align} 
\end{lem}

\begin{proof}
The proof is similar to that of Lemma 5 of \cite{BDH-19}.
By \eqref{eq-2.13-0424} in Lemma \ref{lem-2.8-0424} and \eqref{eq-3.5-0321}, 
we have that 
\begin{align} \label{eq-3.13-0403}
\|g(\nabla u_1)\|_{\mathcal{L}_T^\b} 
 \lesssim  &   T^{\frac{\a-\b-1}2}\|g\|_{C^1} 
(1+ \| u_1\|_{\mathcal{L}_T^\a} ) +\|g\|_{C^1}(1+\| u_0\|_{C^\a})
\\
 \lesssim  &   T^{\frac{\a-\b-1}2}\|g\|_{C^1} 
(1+ \|{\bf u}_1\|_{\a, \b, \ga} )(1+\|X\|_{C^\a})
 +\|g\|_{C^1}(1+\| u_0\|_{C^\a}).\notag
\end{align}
Recalling the assumption on $a$:
$a\in C_b^3(\R)$ and $0<c\leq a(v)\leq C$, we have 
$\frac{1}{a(\nabla u_0^T)}\in C^\b$. More precisely, \eqref{eq-2.11-0424} and \eqref{eq-3.1-430} in Lemma \ref{lem-2.3-0309} yield that 
\begin{align}\label{eq-3.17-0418}
\left\|\frac{1}{a(\nabla u_0^T)}\right\|_{C^\b} \lesssim  &
c^{-2}\| a(\nabla u_0^T) \|_{C^{\b}} 
\lesssim  c^{-2}\| a\|_{C^1} 
(1+\|u_0^T\|_{C^\a} )  
\\
\lesssim & c^{-2}\| a\|_{C^1} 
(1+\|u_0\|_{C^\a} ). \notag 
\end{align}
Then, by the definition of  $\Phi$ and Lemma \ref{lem-5.1-0301}, 
it follows  that 
\begin{align*}
\| v_1'\|_{\mathcal{L}_T^\b} \lesssim &
\left\|\frac{1}{a(\nabla u_0^T)}\right\|_{C^\b} 
\left\{ 
\|g(\nabla u_1)\|_{\mathcal{L}_T^\b} + \|a(\nabla u_1) -a(\nabla u_0^T)
\|_{\mathcal{L}_T^\b} \|u_1' \|_{\mathcal{L}_T^\b}
\right\},  
\end{align*}
which gives  \eqref{eq-2.13-0302} by \eqref{eq-5.8-0301} with $u=u_1$, \eqref{eq-3.13-0403} and \eqref{eq-3.17-0418}.

Next, we give the proof of \eqref{eq-2.18-0304}.
It is enough for us to show the local Lipschitz estimates for
$g(\nabla u)$ and $\big(a(\nabla u)-a(\nabla u_0^T)\big) u'$, respectively.
For the term $g(\nabla u)$, thanks to \eqref{eq-2.14-0424} in Lemma \ref{lem-2.8-0424}, 
we see that
the local Lipschitz estimate  \eqref{eq-3.13-0409} holds also for $g$ instead of $a$.
To deal with $\big(a(\nabla u)-a(\nabla u_0^T)\big) u'$, we set $b_i=a(\nabla u_i)-a(\nabla u_0^T)$ for $i=1,2$. Then it is enough to
estimate $\|b_1u_1'-b_2u_2'\|_{\mathcal{L}_T^\b}$.
Since 
$\|fg\|_{\mathcal{L}_T^\b}\lesssim\|f\|_{\mathcal{L}_T^\b}\|g\|_{\mathcal{L}_T^\b}$ holds for any  $f, g \in \mathcal{L}_T^\b$, 
we have 
\begin{align*}
 \|b_1u_1'-b_2u_2'\|_{\mathcal{L}_T^\b} 
\le & \|b_1\|_{\mathcal{L}_T^\b}\|u_1'-u_2'\|_{\mathcal{L}_T^\b}
+ \|b_1-b_2\|_{\mathcal{L}_T^\b} \|u_2'\|_{\mathcal{L}_T^\b}  \\
\lesssim & T^{\frac{\a-\b-1}2} 
K(\|{\bf u}_1\|_{\alpha,\beta,\gamma}, \|{\bf u}_2\|_{\alpha,\beta,\gamma})  
\|{\bf u}_1- {\bf u}_2 \|_{\a, \b,\ga} (1+\|X\|_{C^\a})^2,
\end{align*}
where we have used \eqref{eq-5.8-0301} and \eqref{eq-3.13-0409}.
Therefore, we have \eqref{eq-2.18-0304}.
\end{proof}

Now we turn to the study of the property of $v^\sharp$. To do it, 
according to our observation in Section \ref{sec-2.1-0307},
 we give the following
Schauder estimate as preparation. We will take $b=a(\nabla u_0^T)$
in the next lemma to show Corollary \ref{cor:5.4} for $v^\sharp$.
\begin{lem}  \label{lem:5.3}  (Schauder estimate)
Let an initial value $f_0\in C^\a$ and a function $b\in C_b^2(\mathbb{T})$, which is
uniformly positive: $b\geq c>0$, be given.  Let $\phi_1 \in C((0,T], C^{\ga-2})$
satisfying
$$
\sup_{0<t \le T} t^{\frac{\ga-\a}2} \|\phi_1(t)\|_{C^{\ga-2}} < \infty,
$$
and $\phi_2 \in C((0,T], C^{\a+\b-2})$ satisfying
\begin{align} \label{eq-3.12-0324}
\sup_{0<t \le T} t^{\frac{\ga -\a}2} 
\|\phi_2(t)\|_{C^{\a+ \b-2}} < \infty,
\end{align}
be given.  Let $f$ be the solution of the parabolic equation
\begin{align}\label{eq-2.21-0309}
\partial_t f - b\De f = \phi_1+\phi_2, \quad f(0) =f_0.
\end{align}
Then, choosing $T>0$ small enough, we have
\begin{align} \label{eq-3.261-430}
\sup_{0<t \le T} & t^{\frac{\ga-\a}2} \|f(t)\|_{C^\ga} +\|f\|_{\mathcal{L}_T^\a}\\
& \lesssim \|f_0\|_{C^\a}
+ \sup_{0<t \le T} t^{\frac{\ga-\a}2} \|\phi_1(t)\|_{C^{\ga-2}} 
+ T^{\frac{\a +\b -\gamma}2}\sup_{0<t \le T} t^{\frac{\ga -\a}2} 
\|\phi_2(t)\|_{C^{\a+\b-2}}   \notag
\end{align}
with an implicit multiplicative positive constant in the right hand side depending 
only on the $C^\a$-norm of $b$ and the lower bound  $c>0$ of $b$.
\end{lem}

\begin{proof}
To show this lemma, we use the semigroup approach   similar to that in \cite{BDH-19}. By the relation
$b\De f= \nabla( b\nabla f) -\nabla b \cdot \nabla f$, the parabolic equation  \eqref{eq-2.21-0309} can be rewritten as 
 \begin{align*}
\partial_t f - \nabla( b\nabla f) = -\nabla b \cdot \nabla f +\phi_1+\phi_2, \quad f(0) =f_0.
\end{align*}
Let $Q_t$ denote the semigroup generated by $\partial_t -\nabla( b\nabla \cdot)$.
Then the solution $f$ of  \eqref{eq-2.21-0309} can be represented 
in the mild form as 
\begin{align} \label{eq-2.22-0309}
f(t) = & Q_t f_0 + \int_0^t Q_{t-s} \phi_1(s)ds + \int_0^t Q_{t-s} \phi_2(s)ds -\int_0^t Q_{t-s} \big(\nabla b \cdot \nabla f(s) \big)ds\\
=: & I_0(t) + I_1(t) + I_2(t) + I_3(t). \notag
\end{align}
It is shown in \cite{BB-16} that $Q_t$ satisfies some important Schauder estimates similar to Lemma \ref{lem-2.3-0309} for  $P_t$; recall Lemma \ref{lem-2.3-0309}-{\rm (iv)}. 
We divide the  proof into three steps.

\noindent  {\it Step 1.} We show that there exists a small enough time horizon $T>0$ such that  
\begin{align} \label{eq-2.23-0309}
\sup_{0<t \le T} t^{\frac{\ga-\a}2} \|f(t)\|_{C^\ga}  \lesssim & \|f_0\|_{C^\a}
+ \sup_{0<t \le T} t^{\frac{\ga-\a}2} \|\phi_1(t)\|_{C^{\ga-2}}  \\
& + T^{\frac{\a +\b -\gamma}2}\sup_{0<t \le T} t^{\frac{\ga -\a}2} 
\|\phi_2(t)\|_{C^{\a+\b-2}}. \notag
\end{align}
This  will be used in both {\it Step 2} and {\it Step 3}.
From \eqref{eq-3.1-430} and {\rm (iii)} in Lemma  \ref{lem-2.3-0309} for  $Q_t$ instead of $P_t$, it follows easily that
\begin{align} \label{eq-2.24-0309}
t^{\frac{\gamma -\a}2 }\|I_0(t)\|_{C^\gamma} \lesssim  \|f_0\|_{C^\a} \  \ \text{and} \ \ 
t^{\frac{\gamma -\a}2 }\|I_1(t)\|_{C^\gamma} \lesssim  \sup_{0<s \leq t}s^{\frac{\gamma -\a}2} \|\phi_1(s)\|_{C^{\gamma -2}}.
\end{align}
Since $\gamma >0>\a +\b -2$, by \eqref{eq-3.1-430}  for $Q_t$, we see that
\begin{align} \label{eq-2.25-0309}
t^{\frac{\gamma -\a}2 }\|I_2(t)\|_{C^\gamma} \leq  
& t^{\frac{\gamma -\a}2 } \int_0^t \|Q_{t-s}\phi_2(s)\|_{C^\gamma}ds \\
\lesssim & t^{\frac{\gamma -\a}2 } \int_0^t  (t-s)^{-\frac{\gamma -(\a +\b -2)}{2}} \|\phi_2(s)\|_{C^{\a +\b -2}}ds \notag \\
\lesssim & t^{\frac{\gamma -\a}2 } \sup_{0<s \leq t}s^{\frac{\gamma -\a}2} \|\phi_2(s)\|_{C^{\a +\b -2}}
 \int_0^t  (t-s)^{-1+ \frac{\a +\b -\gamma}{2}} s^{-\frac{\gamma -\a}{2}}ds \notag \\
\lesssim & T^{\frac{\a+ \b -\ga}2 } \sup_{0<s \leq t}s^{\frac{\gamma -\a}2} \|\phi_2(s)\|_{C^{\a +\b -2}}, \ t\leq T.  \notag
\end{align}
Here the fact that for any $t>0$
\begin{align} \label{eq-2.26-0309}
\int_0^t (t-s)^{p-1} s^{q-1} ds = t^{p+q-1}B(p,q),  \ \ p, q\in (0,1),
\end{align}
and $\a + \b -\gamma, \ga -\a \in (0,1)$
have been used for the last line, where $B(p,q)$ denotes the beta function with parameters $p, q$. 

To evaluate the third term $I_3(t)$,
we first note that for any $0<t\leq T$
\begin{align} \label{eq-2.27-0310}
\sup_{0< s \leq t } s^{\frac{\gamma-\a}{2}}
\|\nabla b \cdot \nabla f(s)\|_{C^{\a-1} }
\lesssim &  \sup_{0< s \leq t }s^{\frac{\gamma-\a}{2}} \|\nabla b\|_{C^{\a-1}}
                \|\nabla f(s)\|_{C^{\gamma -1}} \\
\lesssim &
                 \sup_{0< s \leq t }s^{\frac{\gamma-\a}{2}} \| f(s)\|_{C^{\gamma }}, \notag
\end{align}
where $1<\a<\ga$ and Lemma \ref{lem-5.1-0301} have been used for the first inequality.
Then, by analogous arguments for \eqref{eq-2.25-0309}, we have 
\begin{align*}
t^{\frac{\gamma-\alpha}2}\|I_3(t)\|_{C^\gamma}
&\lesssim t^{\frac{\gamma-\alpha}2}\int_0^t(t-s)^{-\frac{\gamma-\alpha+1}2}\|\nabla b\cdot\nabla f(s) \|_{C^{\alpha-1}}ds \\
&\lesssim T^{\frac{1-\gamma+\alpha}2} \sup_{0< s \leq t }s^{\frac{\gamma-\a}{2}} \| f(s)\|_{C^{\gamma }},  \ t\in [0,T], \notag
\end{align*}
where we have used $1- \ga +\a \in (0, 1)$ and 
\eqref{eq-2.26-0309} for the last inequality.
So, using the relation $1- \ga +\a>0$, we can choose a  small enough $T>0$ such that 
\begin{align}\label{eq-2.28-0310}
t^{\frac{\gamma-\alpha}2}\|I_3(t)\|_{C^\gamma} 
\leq  \frac12 \sup_{0< s \leq t }s^{\frac{\gamma-\a}{2}} \| f(s)\|_{C^{\gamma }},  \ t\in [0,T].
\end{align}
Consequently, by \eqref{eq-2.24-0309}, \eqref{eq-2.25-0309}, \eqref{eq-2.28-0310} and $1- \ga +\a>0$, we obtain 
\eqref{eq-2.23-0309}.

\noindent {\it Step 2.} In this step, we derive the estimate on the norm $\|f\|_{C_TC^\a}.$
By \eqref{eq-3.1-430}  with $\b=0$ in  Lemma  \ref{lem-2.3-0309} for $Q_t$, it is easy to know that $\|I_0(t)\|_{C^\a} \lesssim \|f_0\|_{C^\a}$. 
Using  \eqref{eq-3.1-430}  for $Q_t$ again and the relation $\a- \ga+2 \in (1,2), \ \ga-\a\in (0,1)$, we have
\begin{align} \label{eq-3.31-430}
& \|I_1(t)\|_{C^\a} + \|I_2(t)\|_{C^\a}  \\
 \leq &  \int_0^t \big( \| Q_{t-s} \phi_1(s)\|_{C^\a} + \| Q_{t-s} \phi_2(s)\|_{C^\a} \big) ds \notag\\
\lesssim  & \int_0^t \left( 
(t-s)^{-\frac{\a-\ga+2}{2}} \|\phi_1(s)\|_{C^{\gamma -2 } } + 
(t-s)^{-\frac{2-\b}{2}} \|\phi_2(s)\|_{C^{\a + \b -2} } 
\right) ds \notag \\
\leq  &
\sup_{0<s\leq t} s^{\frac{\gamma -\a}2}  \|\phi_1(s)\|_{C^{\ga -2} } \int_0^t (t-s)^{-\frac{\a-\ga+2}{2}} s^{-\frac{\gamma -\a}{2}} ds 
\notag \\
& +\sup_{0<s\leq t} s^{\frac{\gamma -\a}2}  \|\phi_2(s)\|_{C^{\a + \b -2} } \int_0^t (t-s)^{-\frac{2-\b}{2}} s^{-\frac{\gamma -\a}{2}} ds
\notag \\
\lesssim & \sup_{0<s\leq t} s^{\frac{\gamma -\a}2}  \|\phi_1(s)\|_{C^{\ga -2} }+ T^{\frac{\a +\b -\gamma}2 } \sup_{0<s\leq t} t^{\frac{\gamma -\a}2 }  \|\phi_2(s)\|_{C^{\a + \b -2} }, \ t\leq T.\notag
\end{align}
Using \eqref{eq-3.1-430} for $Q_t$ and \eqref{eq-2.27-0310}, we similarly have 
\begin{align}\label{eq-3.32-430}
\|I_3(t)\|_{C^\alpha}
&\leq \int_0^t \|Q_{t-s} \big(\nabla b \cdot \nabla f(s) \big)\|_{C^\a}ds  \\
&\lesssim \int_0^t(t-s)^{-\frac12}\|\nabla b\cdot\nabla f (s)\|_{C^{\alpha-1}}ds \notag \\
&\lesssim T^{\frac{1-\gamma+\alpha}2}  \sup_{0< s \leq T }s^{\frac{\gamma-\a}{2}} \| f(s)\|_{C^{\gamma }}, \ t\in [0,T]. \notag
\end{align}
From the above estimates and \eqref{eq-2.23-0309}, we obtain that
\begin{align} \label{eq-2.30-0306}
\|f\|_{C_T C^\a} \lesssim & \|f_0\|_{C^\a}
+ \sup_{0<t \le T} t^{\frac{\ga-\a}2} \|\phi_1(t)\|_{C^{\ga-2}}  
 + T^{\frac{\a +\b -\gamma}2}\sup_{0<t \le T} t^{\frac{\ga -\a}2} 
\|\phi_2(t)\|_{C^{\a+\b-2}}. 
\end{align}
\noindent{\it Step 3.} We devote to evaluating $\|f\|_
{C_T^{\frac{\a}2}L^\infty}$.  Let $0\leq s<t \leq T$. Then by \eqref{eq-2.22-0309}, we have
\begin{align*}
& \|f(t)- f(s)\|_{L^\infty}   \\
\leq & \|(Q_{t}- Q_s) f_0\|_{L^\infty}  
   + \left\|\int_0^s \big(Q_{t-r} - Q_{s-r} \big) \phi_1(r)dr \right\|_{L^\infty} 
 + \left\|\int_s^t  Q_{t-r}  \phi_1(r)dr \right\|_{L^\infty} \notag \\
& + \left\|\int_0^s \big(Q_{t-r} -  Q_{s-r} \big) \phi_2(r)dr \right\|_{L^\infty} 
 + \left\|\int_s^t  Q_{t-r}  \phi_2(r)dr \right\|_{L^\infty}  \notag \\
 & + \left\|\int_0^s \big(Q_{t-r} -  Q_{s-r} \big) 
\big(\nabla b  \cdot \nabla f(r) \big) dr \right\|_{L^\infty} 
 + \left\|\int_s^t  Q_{t-r} 
 \big(\nabla b \cdot \nabla f(r) \big) dr\right\|_{L^\infty}  \notag \\
 =: & J_0(s, t) +J_1(s, t) +J_2(s, t) +J_3(s, t) +J_4(s, t) + J_5(s, t) +J_6(s, t). \notag
\end{align*}
By  \eqref{eq-3.5-430} in Lemma \ref{lem-2.3-0309}  for $Q_t$ and the contractivity of the semigroup $Q_t$ on $L^\infty$, 
we have
\begin{align*}
J_0(s, t) \lesssim  \|(Q_{t-s}- {\rm Id}) f_0\|_{L^\infty}   
\lesssim |t-s|^{\frac{\a}2} \|f_0\|_{C^\a}.
\end{align*}
Thanks to \eqref{eq-3.5-430}, the terms $J_i(s, t), i=1,3,5$
can be estimated by repeating essentially the same arguments in {\it Step 2}.
In fact, using  \eqref{eq-3.5-430}, 
we have 
\begin{align*}
J_1(s, t)+ J_3(s, t)
 \lesssim   &  |t-s|^{\frac{\a}{2}}\int_0^s \big(\|Q_{s-r}\phi_1(r)\|_{C^\a} + \|Q_{s-r}\phi_2(r)\|_{C^\a} \big) dr, \\
  J_5(s, t ) \lesssim  &|t-s|^{\frac{\a}{2} }  \int_0^s \|Q_{s-r}\big(\nabla b \cdot \nabla f(r) \big)\|_{C^{\a}} dr.
\end{align*}
Then, by the analogous arguments for \eqref{eq-3.31-430} and \eqref{eq-3.32-430}, we have 
\begin{align*}
J_1(s, t)+ J_3(s, t)
 \lesssim   & |t-s|^{\frac{\a}2 } \Big(
\sup_{0< r \leq T} r^{\frac{\gamma -\a}{2}} \|\phi_1(r)\|_{C^{\gamma-2}} 
+
T^{\frac{\a +\b -\gamma}{2}}\sup_{0< r \leq T} r^{\frac{\gamma -\a}{2}} \|\phi_2(r)\|_{C^{\a +\b-2} } \Big), \\
  J_5(s, t ) \lesssim  &  |t-s|^{\frac\a2}  T^{\frac{1-\gamma+\a}2} 
\sup_{0< r \leq T }r^{\frac{\gamma-\a}{2}} \| f(r)\|_{C^{\gamma }}.
\end{align*}
It is easier to evaluate $J_2(s, t)$ and $J_4(s, t)$ by noting that $\ga -2<0$ and $\a + \b -2<0$ and using \eqref{eq-3.3-430} in 
Lemma \ref{lem-2.3-0309} for $Q_t$. In fact, we have
\begin{align*}
& J_2(s, t) + J_4(s, t) \\
\lesssim &\int_s^t \left( 
|t-r|^{\frac{\gamma -2}2} \|\phi_1(r)\|_{C^{\gamma-2}} + 
|t-r|^{\frac{\a +\b -2}{2}}\|\phi_2(r)\|_{C^{\a +\b-2}} \right) dr \notag \\
\leq &  |t-s|^{\frac{\a}2 }\sup_{0< r \leq t} r^{\frac{\gamma-\a}{2}} \|\phi_1(r)\|_{C^{\gamma-2} }
\int_0^t (t-r)^{-1+ \frac{\gamma -\a}{2}} r^{-\frac{\gamma-\a}{2}}dr\notag \\
&  +|t-s|^{\frac{\a}2 }\sup_{0< r \leq t} r^{\frac{\gamma-\a}{2}} \|\phi_2(r)\|_{C^{\a +\b-2}}
\int_0^t (t-r)^{-1+\frac{\b}{2}} r^{-\frac{\gamma-\a}{2}}dr\notag \\
\lesssim &  |t-s|^{\frac{\a}2 }
\Big(\sup_{0< r \leq T} r^{\frac{\gamma-\a}{2}} \|\phi_1(r)\|_{C^{\gamma-2} }+  T^{\frac{\a +\b -\gamma}{2}}\sup_{0< r \leq T} r^{\frac{\gamma-\a}{2}} \|\phi_2(r)\|_{C^{\a +\b-2}} \Big).\notag 
\end{align*}
For the term $J_6(s, t)$, 
thanks to the  contractivity  of the semigroup $Q_t$ on $L^\infty$ and the fact $C^{\a-1}\subset L^\infty$, by \eqref{eq-2.27-0310}, we have
\begin{align*}
J_6( s, t ) \leq &\int_s^t \left\| Q_{t-r} \big( \nabla b\cdot \nabla f (r)\big)\right\|_{L^\infty} dr \\
\lesssim & \int_s^t \|\nabla b \cdot \nabla f (r)\|_{C^{\a-1}} dr \\
\lesssim &|t-s|^{\frac{\alpha}2}  T^{1-\frac{\ga}{2}} 
\sup_{0<r\leq T}r^{\frac{\ga -\a}2 }  \|f(r)\|_{C^\ga}.
\end{align*}
Noting that $1-\frac{\ga}{2}>0$ and using the above estimates in this step together with  \eqref{eq-2.23-0309}, 
we have that 
\begin{align} \label{eq-2.37-0306}
\|f\|_{C_T^{\frac{\a}2}L^\infty} \lesssim \|f_0\|_{C^\a}
+ \sup_{0<t \le T} t^{\frac{\ga-\a}2} \|\phi_1(t)\|_{C^{\ga-2}}   + T^{\frac{\a +\b -\gamma}2}\sup_{0<t \le T} t^{\frac{\ga -\a}2} 
\|\phi_2(t)\|_{C^{\a+\b-2}}
\end{align}
holds for small enough $T>0$.

Consequently, we  obtain the desired result \eqref{eq-3.261-430} by \eqref{eq-2.23-0309},  \eqref{eq-2.30-0306} and \eqref{eq-2.37-0306}. Hence the proof is completed.
\end{proof}

Applying Lemma \ref{lem:5.3} to the equation \eqref{eq:3.L0phi} below, which is
essentially \eqref{eq-2.8-0323} in our fixed point problem, we have the following estimate
for $v^\sharp:= v- \bar{\Pi}_{v'}X$ for $v$ determined by \eqref{eq-2.8-0323} and
\eqref{eq-2.7-0323}.

\begin{cor}  \label{cor:5.4}
Let $\phi_1 \in C((0,T], C^{\ga-2})$ and 
$\phi_2 \in C((0,T], C^{\a+\b-2})$
be functions as in Lemma \ref{lem:5.3}.  For given
$z_0\in C^\a$ and $z'\in \mathcal{L}_T^\b$, let $z$ be the solution of the equation
\begin{equation}  \label{eq:3.L0phi}
\big(\partial_t - a(\nabla u_0^T) \De \big) z
 = \Pi_{a(\nabla u_0^T)z'} \xi + \phi_1+\phi_2, \quad z(0) =z_0.
\end{equation}
Then, $(z,z') \in {\bf C}_{\a,\b,\ga}(X)$ and we have the estimate
\begin{align} \label{eq-2.21-0307}
 & \sup_{0<t \le T}  t^{\frac{\ga-\a}2} \|z^\sharp(t)\|_{C^\ga} 
+\|z^\sharp\|_{\mathcal{L}_T^\a}\\
 \lesssim & \|z_0\|_{C^\a} + \|z'(0)\|_{L^\infty} \|\xi\|_{C^{\a -2}}
+ T^{\frac{\a + \b -\ga}2} (1+\|u_0\|_{C^\a}) \|z'\|_{\mathcal{L}_T^\b} \|\xi\|_{C^{\a -2}}
\notag \\
& \quad + \sup_{0<t \le T} t^{\frac{\ga-\a}2} \|\phi_1(t)\|_{C^{\ga-2}} 
+ T^{\frac{\a + \b -\ga}2}\sup_{0<t \le T}t^{\frac{\ga -\a}2} 
\|\phi_2(t)\|_{C^{\a+\b-2}}\notag
\end{align}
with a multiplicative positive constant in the right hand side depending 
only on the $C^\a$-norm of $u_0$.

If $(y,y') \in {\bf C}_{\a,\b,\ga}(X)$ is associated similarly to
another set of data $\psi_1, \psi_2, y_0$ and $y'$ with $y$ the 
solution of the equation
$$
\big(\partial_t - a(\nabla u_0^T) \De \big) y
 = \Pi_{a(\nabla u_0^T)y'} \xi + \psi_1+\psi_2, \quad 
 y(0) =y_0 \in C^\a,
$$
then we have the Lipschitz continuity bound 
\begin{align}\label{eq-2.22-0307}
& \sup_{0<t \le T}  t^{\frac{\ga-\a}2} \|z^\sharp(t) 
- y^\sharp(t)\|_{C^\ga} 
+\|z^\sharp-y^\sharp\|_{\mathcal{L}_T^\a}\\
 \; \lesssim &  \|z_0-y_0\|_{C^\a} + \|z'(0)-y'(0)\|_{L^\infty} \|\xi\|_{C^{\a -2}}
+ T^{\frac{\a + \b -\ga}2} (1+\|u_0\|_{C^\a}) \|z'-y'\|_{\mathcal{L}_T^\b} \|\xi\|_{C^{\a -2}} \notag \\
& \quad + \sup_{0<t \le T} t^{\frac{\ga-\a}2} 
\|\phi_1(t)-\psi_1(t) \|_{C^{\ga-2}}
+ T^{\frac{\a + \b -\ga}2}\sup_{0<t \le T} t^{\frac{\ga -\a}2} 
\|\phi_2(t)-\psi_2(t)\|_{C^{\a +\b -2}} \notag
\end{align}
with a multiplicative positive constant in the right hand side depending 
only on the $C^\a$-norm of $u_0$. 
\end{cor}
\begin{proof}
 Set
$$
z^\sharp = z- \bar{\Pi}_{z'}X.
$$
Then, to prove the first part of this lemma, it is enough for us 
to show \eqref{eq-2.21-0307}.
Recalling that $\mathcal{L}^0 =\partial_t -a(\nabla u_0^T) \Delta$ and  noting that 
\begin{align*}
\mathcal{L}^0(\bar{\Pi}_{z'}X) = \left\{ \mathcal{L}^0(\bar{\Pi}_{z'}X) - \Pi_{a(\nabla u_0^T)z'} (-\Delta X)\right\} +\Pi_{a(\nabla u_0^T)z'} (-\Delta X), 
\end{align*}
we have that
\begin{align}\label{eq-2.23-0307}
\mathcal{L}^0z^\sharp =& \phi_1 + \phi_2 -\big\{ \mathcal{L}^0(\bar{\Pi}_{z'}X) - \Pi_{a(\nabla u_0^T)z'} (-\Delta X)\big\}
\end{align}
with the initial condition $z^\sharp (0)= z_0 - (\bar\Pi_{z'}X)(0)
= z_0 -\Pi_{z'(0)}X \in {C}^\alpha$. Note that 
$(\bar\Pi_{z'}X)(0)=\Pi_{z'(0)}X$ because $X$ is independent of time $t$, see p.\ 45 of \cite{BDH-19}.
We write  $\Pi_{z'(0)}X$  for $ (\bar\Pi_{z'}X)(0)$ whenever $z' \in \mathcal{L}_T^\b$ in this paper.

Since  $z' \in \mathcal{L}_T^\b$,
by the intertwining continuity estimate in  Lemma \ref{lem-2.4-0319}, we see  that
\begin{align*} 
\|\mathcal{L}^0(\bar{\Pi}_{z'}X) - \Pi_{a(\nabla u_0^T)z'}(-\Delta X)
\|_{C_TC^{\a+ \b-2}}
\lesssim (1+ T^{-\frac{\gamma -\a}{2}}\|u_0\|_{C^\a} )\|z' \|_{\mathcal{L}_T^\b}\|X\|_{C^\a},
\end{align*}
which particularly implies that 
\begin{align} \label{eq-3.29-0321}
\sup_{0<t \leq T} t^{\frac{\ga -\a}{2}}
\|\big(\mathcal{L}^0(\bar{\Pi}_{z'}X) - \Pi_{a(\nabla u_0^T)z'}(-\Delta X) \big)(t)\|_{C^{\a+ \b-2}}
\lesssim (1+ \|u_0\|_{C^\a} )\|z' \|_{\mathcal{L}_T^\b}\|X\|_{C^\a}.
\end{align}
So, the condition \eqref{eq-3.12-0324} in Lemma \ref{lem:5.3} holds for
$\phi_2 -\big\{ \mathcal{L}^0(\bar{\Pi}_{z'}X) - \Pi_{a(\nabla u_0^T)z'} (-\Delta X)\big\}$. 

Now, let $b(x)=a(\nabla u_0^T(x))$.  By the assumption on $a$, we know
$b$ satisfies the assumption of Lemma \ref{lem:5.3}.  
Then, \eqref{eq-2.21-0307} is obtained by  Lemma \ref{lem:5.3}, \eqref{eq-3.29-0321} and noting that $\|z^\sharp (0)\|_{C^\a } \lesssim \|z_0\|_{C^\a} + \|z'(0)\|_{L^\infty}\|X\|_{C^\a}$.

Next, let us give the proof of \eqref{eq-2.22-0307}, which is shown similarly to \eqref{eq-2.21-0307} with the help of  Lemma \ref{lem:5.3}. In fact, by the assumptions on $(z,z')$ and 
$(y, y')$, we deduce that 
\begin{align*}
\big(\partial_t - a(\nabla u_0^T) \De \big) (y-z)
 = \Pi_{a(\nabla u_0^T)(y'-z')} \xi + (\psi_1-\phi_1)+(\psi_2-\phi_2)
\end{align*} 
with $ y(0)-z(0) =y_0-z_0 \in C^\a$.

Setting $y^\sharp = y- \bar{\Pi}_{y'}X$ and then similarly to \eqref{eq-2.23-0307}, we have that 
\begin{align*}
\mathcal{L}^0(y^\sharp -z^\sharp) = &(\psi_1-\phi_1)+(\psi_2-\phi_2)
-\big\{ \mathcal{L}^0(\bar{\Pi}_{y'}X) - \Pi_{a(\nabla u_0^T)y'} (-\Delta X)\big\} 
\\
&+\big\{ \mathcal{L}^0(\bar{\Pi}_{z'}X) - \Pi_{a(\nabla u_0^T)z'} (-\Delta X)\big\}\\
=&(\psi_1-\phi_1)+(\psi_2-\phi_2)
-\big\{ \mathcal{L}^0(\bar{\Pi}_{y'-z'}X) - \Pi_{a(\nabla u_0^T)(y'-z')} (-\Delta X)\big\}
\end{align*}
with the initial condition $ 
(y^\sharp -z^\sharp) (0)=(y_0-z_0) -\Pi_{(y'-z')(0)}X \in {C}^\alpha$.
Hence,   the desired result \eqref{eq-2.22-0307} can be easily obtained by Lemma \ref{lem:5.3}.
\end{proof}

The next proposition shows the  local  growth and local Lipschitz continuity of the remainder
term $\zeta$ defined by \eqref{eq-2.8-0307}.  The proof is
given in Section \ref{sec-4-0406}.   
 
\begin{prop} \label{lem:Lip-zeta}
For any ${\bf u}_1 = (u_1,u_1'), {\bf u}_2 = (u_2,u_2') \in \mathcal{B}_T(\la)$,  
we have
\begin{align} 
& \sup_{0<t \le T} t^{\frac{\ga -\a}2} 
\|\zeta({\bf u}_1) (t)\|_{C^{\a +\b -2}} 
\lesssim  K(\|{\bf u}_1 \|_{\a, \b, \ga})\tilde{K}_1(X, \xi),    \label{eq-3.39-0409}\\
& \sup_{0<t \le T} t^{\frac{\ga -\a}2} 
\|\big(\zeta({\bf u}_1) -\zeta({\bf u}_2) \big)(t)\|_{C^{\a +\b -2}} 
\label{eq-3.40-0409}
\\
\lesssim & K(\|{\bf u}_1\|_{\a, \b, \ga}, \|{\bf u}_2\|_{\a, \b, \ga})
\|{\bf u}_1 -{\bf u}_2\|_{\a, \b, \ga}
\tilde{K}_2(X, \xi),   \notag  
\end{align}
where 
\begin{align*}
\tilde{K}_1(X, \xi) = & (1+\|\xi\|_{C^{\a -2}})^2(\|\xi\|_{C^{\a-2}}
+ \|\Pi(\nabla X, \xi)\|_{C^{2\a -3}}), \\
\tilde{K}_2(X, \xi) = & ( 1+\|\xi\|_{C^{\a -2}})^2
\big((1+\|\xi\|_{C^{\a-2}})^2
+ \|\Pi(\nabla X, \xi)\|_{C^{2\a -3}} \big).
\end{align*}
\end{prop}

\subsection{Proof of Theorem \ref{thm-2.1-0304}}

Based on these preparations, we are at the position to give the proof of Theorem \ref{thm-2.1-0304}.

First, we give the proof of (i).
Let us recall that for any  ${\bf u} =(u, u') \in \mathcal{B}_T(\lambda)$, the map  $\Phi({\bf u}) =\Phi(u, u')$ is defined by \eqref{eq-2.4-0304}.
Using  the notations $\e_1(u, u')$ and $\e_2(u, u')$ introduced in Lemma \ref{lem:5.1}, we can rewrite
$\mathcal{L}^0v$ in \eqref{eq-2.8-0323}  as the following: 
\begin{align} \label{eq-2.4-0304-b}
& \mathcal{L}^0v= \Pi_{a(\nabla u_0^T)v'} \xi+\e_1(u, u')+ 
\e_2(u, u') + \zeta, 
\end{align}
with the initial value $v(0)=u_0 \in C^\a$.

We first show that $\Phi$ maps $\mathcal{B}_T(\lambda)$ into itself,
if we choose $\lambda >0$ sufficiently large and $T>0$ small enough.
We have $v' \in \mathcal{L}_T^\b$ for $v'$ determined by 
\eqref{eq-2.7-0323} by Lemma \ref{lem:5.2}. 
Since $\b<\a -1$, Lemmas 2.1-{\rm (i)} and \ref{lem:5.1} yield  that 
\begin{align} \label{eq-3.30-0324}
\sup_{0<t\leq T} t^{\frac{\gamma-\a}{2}}\|\e_1(u,u')\|_{C^{\a +\b-2}}
\lesssim K(\|{\bf u}\|_{\a,\b,\ga})(1+\|\xi\|_{C^{\a -2}})\|\xi\|_{C^{\a -2}}.
\end{align} 

Now, let $\phi_1=  \e_2(u, u')$ and  $\phi_2= \e_1(u, u')+ \zeta$.
By  Lemma \ref{lem:5.1}, \eqref{eq-3.30-0324} and \eqref{eq-3.39-0409},  we see that $\phi_1$ and $\phi_2$ satisfy 
the assumptions in Corollary \ref{cor:5.4}.
In addition, according to the definition of $v'$, we have $v'(0) =u'(0)\in L^\infty$, by recalling that $u'(0)=\frac{g(\nabla u_0)}{a(\nabla u_0)}$.
Applying Corollary \ref{cor:5.4} to \eqref{eq-2.4-0304-b}, and then using  Lemmas \ref{lem:5.1},  \ref{lem:5.2} and \eqref{eq-3.39-0409},
we have 
\begin{align} \label{eq-3.31-0324}
 & \sup_{0<t \le T}  t^{\frac{\ga-\a}2} \|v^\sharp(t)\|_{C^\ga} 
+\|v^\sharp\|_{\mathcal{L}_T^\a}\\
 \lesssim & \|u_0\|_{C^\a} + \|v'(0)\|_{L^\infty} \|\xi\|_{C^{\a -2}}
+ T^{\frac{\a + \b -\ga}2} (1+\|u_0\|_{C^\a}) \|v'\|_{\mathcal{L}_T^\b} \|\xi\|_{C^{\a -2}}
\notag \\
& + \sup_{0<t \le T} t^{\frac{\ga-\a}2} \|\phi_1(t)\|_{C^{\ga-2}} 
+ T^{\frac{\a + \b -\ga}2}\sup_{0<t \le T} t^{\frac{\ga -\a}2} 
\big(\|\e_1(u, u')\|_{C^{\a +\b -2}} + \|\zeta(t)\|_{C^{\a+\b-2}} \big)\notag\\
 \lesssim & T^{\frac{\a + \b -\ga}2} K(\|{\bf u}\|_{\a, \b, \ga})\tilde{K}_1(X, \xi) 
 +K_0   (1+\|\xi\|_{C^{\a -2}}),
\notag
\end{align}
where $\tilde{K}_1(X, \xi)$ is the constant defined in Proposition 
\ref{lem:Lip-zeta},  and the relation $0<\a+\b-\ga<\a-\b-1$ has been 
used to obtain the order $\frac{\a + \b -\ga}2$ of $T$ in the last inequality. 

Recall that  $\|\Phi({\bf u})\|_{\a, \b, \ga} =\|(v, v')\|_{\a, \b, \ga}
=\|v'\|_{\mathcal{L}_T^\b} + \|v^\sharp\|_{\mathcal{L}_T^\a}
+ \sup_{0<t\le T} t^{\frac{\ga-\a}2} \|v^\sharp(t)\|_{C^{\ga}}.$ 
Therefore, thanks to Lemma  \ref{lem:5.2} and \eqref{eq-3.31-0324},
we obtain
\begin{align}\label{eq-3.38-0417}
\|\Phi({\bf u})\|_{\a, \b, \ga} \lesssim
 T^{\frac{\a + \b -\ga}2} K(\|{\bf u}\|_{\a, \b, \ga})\tilde{K}_1(X, \xi) 
 +K_0 (1+\|\xi\|_{C^{\a -2}}), 
\end{align}
where we have used the relation $\a+\b-\ga<\a-\b-1$ again.
Consequently, noting that $\a + \b -\ga >0$,  we can choose sufficient large  $\lambda>0$ and small enough $T>0$ such that 
$\Phi$ maps $\mathcal{B}_T(\lambda)$ into itself.

Now we give the proof of the  contractive property of $\Phi$ on $\mathcal{B}_T(\lambda)$. We use the same 
notations introduced in Lemmas \ref{lem:5.1} and  \ref{lem:5.2}.
Note that for
any  ${\bf u}_1 := (u_1,u_1') \in \mathcal{B}_T(\la)$ and  ${\bf u}_2 := (u_2,u_2') \in \mathcal{B}_T(\la)$, we have $u_1(0)=u_2(0)=u_0$ and $u_1'(0)=u_2'(0)$. So, the definition of $\Phi$ implies that $v_1(0)=v_2(0)=u_0$ and $v_1'(0)=v_2'(0)$. Applying  \eqref{eq-2.22-0307} with this fact and repeating essentially the same arguments for \eqref{eq-3.31-0324},  we have
 \begin{align}\label{eq-3.33-0324}
& \sup_{0<t \le T}  t^{\frac{\ga-\a}2} \|v_1^\sharp(t) 
- v_2^\sharp(t)\|_{C^\ga} 
+\|v_1^\sharp-v_2^\sharp\|_{\mathcal{L}_T^\a}\\
  \lesssim &  \sup_{0<t \le T} t^{\frac{\ga-\a}2} 
\|\big(\e_2({\bf u}_1)-\e_2{(\bf u}_2) \big)(t) \|_{C^{\ga-2}} + T^{\frac{\a+\b-\ga}2} (1+\|u_0\|_{C^\a}) 
\|v_1'-v_2'\|_{\mathcal{L}_T^\b} \|\xi\|_{C^{\a -2}} \notag \\
& + T^{\frac{\a + \b -\ga}2}\sup_{0<t \le T} t^{\frac{\ga -\a}2} 
\Big(
\|\big(\e_1({\bf u}_1)-\e_1({\bf u}_2) \big)(t)\|_{C^{\a +\b -2}} +
\|\big(\zeta({\bf u}_1) -\zeta({\bf u}_2) \big)(t)\|_{C^{\a +\b -2}} 
\Big)  \notag \\
 \lesssim & T^{\frac{\a + \b -\ga}2} 
K(\|{\bf u}_1\|_{\a, \b, \ga}, \|{\bf u}_2\|_{\a, \b, \ga}) 
\|{\bf u}_1 -{\bf u}_2\|_{\a, \b, \ga}\tilde{K}_2(X,\xi),\notag
\end{align}
where $\tilde{K}_2(X,\xi)$ is the constant defined in defined in Proposition \ref{lem:Lip-zeta}. 
Here, for the second inequality, we have used Lemma \ref{lem:5.1} for the terms
involving $\e_1$ (note $\a+\b-2<2\a-3$) and $\e_2$ (note $\a+\b-\ga<\a-\b-1$), Lemma \ref{lem:5.2}
for $\|v_1'-v_2'\|_{\mathcal{L}_T^\b}$ and \eqref{eq-3.40-0409} in Proposition \ref{lem:Lip-zeta} for 
$\zeta=\zeta(u, u')$.  Finally, by \eqref{eq-2.18-0304} in  Lemma \ref{lem:5.2} 
again and \eqref{eq-3.33-0324}, we have that
\begin{align}\label{eq-3.40-0417}
\| \Phi( {\bf u}_1) - \Phi({\bf u}_2)\|_{\a, \b, \ga}
\lesssim T^{\frac{\a + \b -\ga}2} 
K(\|{\bf u}_1\|_{\a, \b, \ga}, \|{\bf u}_2\|_{\a, \b, \ga}) 
\|{\bf u}_1 -{\bf u}_2\|_{\a, \b, \ga}\tilde{K}_2(X,\xi),
\end{align}
which clearly implies that 
 the map $\Phi$ inherits
its contractivity on $\mathcal{B}_T(\lambda)$ if $T>0$ is chosen sufficient small. 
This concludes the proof of the assertion {\rm (i)} of the theorem.

In the end, let us give the proof of {\rm (ii)}. Let $\lambda$ and $T$ be the chosen values in the proof of {\rm (i)}.
We first show the contractive property of  $\Phi$  is locally uniform in  the  enhanced noise $\hat{\xi}=(\xi, \Pi(\nabla X, \xi))\in C^{\a-2}\times C^{2\a -3}$. According to the definition 
of $\tilde{K}_2(X,\xi)$, we easily observe that $\tilde{K}_2(X,\xi)$ is
 bounded from above by an increasing polynomial $P$  of the norm of $\hat{\xi}$.  Then, 
for  all $\hat{\xi}\in B_r:=\{\hat{\xi}=(\xi, \Pi (\nabla X, \xi)): 
\|\xi\|_{C^{\a -2}}+ \|\Pi (\nabla X, \xi)\|_{C^{2\a -3}}\leq r\}, r>0$,
 by \eqref{eq-3.40-0417}, we have
$\| \Phi( {\bf u}_1) - \Phi({\bf u}_2)\|_{\a, \b, \ga}$ is controlled  by $T^{\frac{\a + \b -\ga}2} 
K(\|{\bf u}_1\|_{\a, \b, \ga}, \|{\bf u}_2\|_{\a, \b, \ga}) 
\|{\bf u}_1 -{\bf u}_2\|_{\a, \b, \ga}P(r)$. Therefore, 
using the fact that $\|{\bf u}_i\|_{\a, \b, \ga}\leq \lambda, i=1,2$ 
and 
$K(\|{\bf u}_1\|_{\a, \b, \ga},$ $\|{\bf u}_2\|_{\a, \b, \ga})$ does 
not depend on $X, \xi$ and $\Pi(\nabla X, \xi)$, we obviously
 see the map $\Phi$ is locally Lipschitz on $\mathcal{B}_T(\lambda)$ for all 
$\hat{\xi} \in B_r$ by making $T>0$ small enough if necessary. 
The important observation is that $T$ can be chosen as the function of $r$ for all $\hat{\xi} \in B_r$. Therefore, the desired 
result is proved.

Next, let us show the remaining part, that is, 
$\Phi$ depends continuously on $\hat{\xi}$ and the unique fixed point 
of $\Phi$ in  $\mathcal{B}_T(\la)$ is also continuous in $\hat{\xi}$. 
In order to do it, it is more convenient to take the pair $\tilde{\bf u}=(u',u^\sharp)$
as a variable instead of ${\bf u} = (u,u')$ and regard $\Phi=\Phi(\tilde{\bf u},\hat{\xi})$
as a map from $\widetilde{\bf C}_{\a, \b, \ga} \times ( C^{\a-2} \times C^{2\a -3} )$ to 
$\widetilde{\bf C}_{\a, \b, \ga}$, where  $\widetilde{\bf C}_{\a, \b, \ga} =
\{\tilde{\bf u} \equiv(u', u^\sharp)\in \mathcal{L}_T^\b \times \mathcal{L}_T^\a;
\|\tilde{\bf u}\|_{\a,\b,\ga}<\infty\}$.  Note that the norm 
$\|\cdot\|_{\a,\b,\ga}$ in \eqref{eq:2-abc} is defined essentially for
$\tilde{\bf u}$ and all estimates we obtained in ${\bf u}$ is the same
in $\tilde{\bf u}$.  In particular, the space $\widetilde{\bf C}_{\a, \b, \ga}$ does not 
depend on the noise $\xi$ or $X$.  

Then, thanks to the implicit
function theorem, see Proposition C.1.1 of \cite{DaZ} and the
locally uniform contractivity of  $\Phi$ in $\tilde{\bf u}$, it is enough for us to show  the continuity
of $\Phi$ in $(\tilde{\bf u}, \hat{\xi})$. In fact, we can show that $\Phi$ is 
locally Lipschitz continuous in $(\tilde{\bf u}, \hat{\xi})$. For any two elements 
$(\tilde{\bf u}_1, \hat{\xi}_1)$ and $(\tilde{\bf u}_2, \hat{\xi}_2)$, we see
$
\|\Phi(\tilde{\bf u}_1, \hat{\xi}_1) - \Phi(\tilde{\bf u}_2, \hat{\xi}_2)\|_{\a, \b, \ga}
$
is bounded by
\begin{align*}
\|\Phi(\tilde{\bf u}_1, \hat{\xi}_1) - \Phi(\tilde{\bf u}_2, \hat{\xi}_1)\|_{\a, \b, \ga} +
\|\Phi(\tilde{\bf u}_2, \hat{\xi}_1) - \Phi(\tilde{\bf u}_2, \hat{\xi}_2)\|_{\a, \b, \ga}.
\end{align*}
According to \eqref{eq-3.40-0417}, we only have to 
study the second term. However, thanks to the multilinear property of $\e_1$ and $\zeta$ in the variable
$(X, \xi, \Pi(\nabla X, \xi))$, 
by analogous, but simpler, arguments for \eqref{eq-3.38-0417},
we have 
\begin{align*}
\|\Phi(\tilde{\bf u}_2, \hat{\xi}_1) - \Phi(\tilde{\bf u}_2, \hat{\xi}_2)\|_{\a, \b, \ga}
\lesssim & 
T^{\frac{\a + \b -\ga}2} K(\|\tilde{\bf u}_2\|_{\a, \b, \ga})(1+\|\xi_1\|_{C^{\a-2}} +\|\xi_2\|_{C^{\a -2}} )^3\\
& \times
\big( \|\xi_1- \xi_2\|_{C^{\a -2}}
+ \|\Pi(\nabla X_1, \xi_1) -\Pi(\nabla X_2, \xi_2)\|_{C^{2\a-3}} 
\big).
\end{align*}
Indeed, the cubic power of $\|\xi_i\|_{C^{\a-2}}, i=1,2,$ in the Lipschitz coefficient
in the above estimate arises from the 
terms $P_g$ in $A_1$ and $P_a$ in $A_3$ of $\Phi$ as computed in 
Lemmas \ref{lem-A1} and \ref{lem-A3}, respectively, later and this is reflected in $\tilde K_2(X,\xi)$ in
Proposition \ref{lem:Lip-zeta}, which involves the fourth power of $\|\xi\|_{C^{\a-2}}$.
All other terms in $\Phi$ have lower orders as we can see from Lemmas \ref{lem:5.1}, 
\ref{lem:5.2},  Proposition \ref{lem:Lip-zeta} and \eqref{eq-3.31-0324}.
Consequently, the proof is completed.
\qed

\section{Proof of Proposition \ref{lem:Lip-zeta}} \label{sec-4-0406}

This section is devoted to the proof of   Proposition 
\ref{lem:Lip-zeta}. In order to do it,  we  study  each term appeared 
in $\zeta$ given in \eqref{eq-2.8-0307} separately and divide it into 
four lemmas.

Let us  denote by $A_0$ all of the  terms in $\zeta$ 
except the three terms $A_1$, $A_2$ and $A_3$, that is,
\begin{align*}
A_0= A_0(u, u')=& \Pi_\xi g(\nabla u) +g'(\nabla u) u' \Pi(\nabla X, \xi ) \\
& 
 + \Pi(a(\nabla u_0^T), \bar\Pi_{u'}\xi ) - a'(\nabla u) (u')^2 \Pi(\nabla X, \xi ),
\end{align*}
so that $\zeta=A_0+A_1+A_2-A_3$.
We first show  $A_0$ has the desired estimates.
\begin{lem} \label{lem-A0}
For any ${\bf u}_1 = (u_1,u_1'), {\bf u}_2 = (u_2,u_2') \in \mathcal{B}_T(\la)$,  
we have
\begin{align} 
 & \sup_{0<t \le T} t^{\frac{\ga -\a}2} 
\|A_0({\bf u}_1) (t)\|_{C^{\a +\b -2}}  \label{eq-4.1-0415} \\
\lesssim &  K(\|{\bf u}_1\|_{\a, \b, \ga})
(1+\|\xi\|_{C^{\a-2}})(\|\xi\|_{C^{\a-2} } 
+ \|\Pi(\nabla X, \xi)\|_{C^{2\a -3}}),    
\notag \\
& \sup_{0<t \le T} t^{\frac{\ga -\a}2}  
\|\big(A_0({\bf u}_1) -A_0({\bf u}_2) \big)(t)\|_{C^{\a +\b -2}}  \label{eq-4.2-0415}  \\
\lesssim &  K(\|{\bf u}_1 \|_{\a, \b, \ga}, \|{ \bf u}_2\|_{\a, \b, \ga})
\|{\bf u}_1 -{\bf u}_2\|_{\a, \b, \ga}  \notag   \\
& \hskip 20mm \times  (1+\|\xi\|_{C^{\a-2}})^2 
(\|\xi\|_{C^{\a-2} } + \|\Pi(\nabla X, \xi)\|_{C^{2\a -3}}).     \notag
\end{align}
\end{lem}
\begin{proof} 
We first give the proof of \eqref{eq-4.1-0415}. Applying Lemma \ref{lem-5.1-0301} together with \eqref{eq-3.4-0424} for $g$ instead of $g'$, 
we get
\begin{align*}
\|\Pi_\xi g(\nabla u_1)  (t)\|_{C^{\a + \b -2}}
\lesssim \|g(\nabla u_1)(t)\|_{C^{\b} } \|\xi \|_{C^{\a-2}}
\lesssim K(\|{\bf u}_1\|_{\a, \b, \ga}) (1+\|X\|_{C^\a}) \|\xi \|_{C^{\a-2}},
\end{align*}
where we have used  the relation $0<\a+ \b-2<\b$
for the first inequality.

Recall  that $0<2\a +\b -3< \b$ and $\Pi(\nabla X, \xi ) \in C^{2\a -3}$ is assumed. Applying now Lemma \ref{lem-5.1-0301}  and \eqref{eq-3.4-0424}, we obtain 
\begin{align*}
\|\big( g'(\nabla u_1) u_1' \Pi(\nabla X, \xi ) \big)(t)\|_{C^{2\a  -3}}
\lesssim & \|g'(\nabla u_1)(t)\|_{C^{\b}}\|u_1'(t)\|_{C^\b}\|\Pi(\nabla X, \xi )\|_{C^{2\a -3}} 
\\
\lesssim &K(\|{\bf u}_1\|_{\a, \b, \ga}) (1+\|X\|_{C^\a}) \|\Pi(\nabla X, \xi )\|_{C^{2\a -3}}.
\end{align*}
Thanks to \eqref{eq-3.4-0424}, the similar arguments give that
\begin{align*}
\|\big(a'(\nabla u_1) (u_1')^2  \Pi(\nabla X, \xi ) \big(t)\|_{C^{2\a  -3}}
\lesssim & K(\|{\bf u}_1\|_{\a, \b, \ga}) (1+\|X\|_{C^\a}) 
\|\Pi(\nabla X, \xi )\|_{C^{2\a -3}}.
\end{align*}
By  \eqref{eq:2.18} and $\ga -\a>0$, we easily have
\begin{align*}
\sup_{0< t \leq T} t^{\frac{\ga -\a}{2}}\|\Pi(a(\nabla u_0^T), \bar\Pi_{u'_1}\xi )(t)\|_{C^{\a +\b -2} }
\lesssim & \|a\|_{C^1}(1+ \|u_0\|_{C^\a})\|{\bf u}_1\|_{\a, \b, \ga}\|\xi\|_{C^{\a -2}}.
\end{align*}
Thus, by the above estimates together with $\ga -\a>0$, 
we obtain the desired result \eqref{eq-4.1-0415};
recall Remark \ref{rem:3.1-A} and the same applies hereafter.

Next, we give the proof  of \eqref{eq-4.2-0415}. 
For the term $\Pi_\xi g(\nabla u)$, 
by \eqref{eq-3.8-03038} for $g$ and  the bilinearity of the paraproduct,  we have
\begin{align*}
& \| \Pi_\xi g(\nabla u_1)(t) - \Pi_\xi g(\nabla u_2)(t)\|_{C^{\a +\b -2}}  
\\
& \lesssim  \| g(\nabla u_1)(t) - g(\nabla u_2)(t)\|_{C^{\b}}  \| \xi\|_{C^{\a-2}} \\
& \lesssim K(\|{\bf u}_1\|_{\a, \b, \ga})
\|{ \bf u}_1- { \bf u}_2\|_{\a, \b, \ga}(1+\|X\|_{C^\a})^2\| \xi\|_{C^{\a-2}}, \ t\in [0,T]. 
\end{align*}

For the term $g'(\nabla u) u' \Pi(\nabla X, \xi )$, noting that 
$\Pi(\nabla X, \xi )\in C^{2\a-3}$,  $2\a-3+\b>0$
 and using the estimate for the product $uv$
given in Lemma \ref{lem-5.1-0301}, we may give the estimate on
$$
\| \big(g'(\nabla u_1) u_1' - g'(\nabla u_2) u_2'\big)(t)\|_{C^{\b}}.
$$
But, this is bounded again by the estimate for the product in Lemma \ref{lem-5.1-0301} as
\begin{align*}
&  \lesssim \|g'(\nabla u_1)(t)\|_{C^\b} \| u_1'(t) - u_2'(t)\|_{C^\b}
+ \|\big(g'(\nabla u_1(t)) - g'(\nabla u_2) \big)(t)\|_{C^\b} \| u_2'(t)\|_{C^\b}  \\
& \lesssim K(\|{\bf u}_1\|_{\a, \b, \ga}, \|{\bf u}_2\|_{\a, \b, \ga})
\|{\bf u}_1 -{\bf u}_2\|_{\a, \b, \ga}(1+\|X\|_{C^\a})^2
\end{align*}
for all $t\in [0,T]$, where 
\eqref{eq-3.4-0424} and \eqref{eq-3.8-03038}
have been used for the last inequality.

The term $a'(\nabla u) (u')^2 \Pi(\nabla X, \xi )$ is treated similarly, for example,
by estimating as
\begin{align*}
 \| (u_1')^2(t) - (u_2')^2(t)\|_{C^\b}
 \lesssim &  \big(\| u_1'(t)\|_{C^\b} + \|u_2'(t)\|_{C^\b}\big) 
 \| u_1'(t) - u_2'(t)\|_{C^\b}\\
\lesssim  &
 (\|{\bf u}_1\|_{\a, \b, \ga} + \|{\bf u}_2\|_{\a, \b, \ga})
\|{\bf u}_1 -{\bf u}_2\|_{\a, \b, \ga} 
\end{align*}
for all $t\in [0,T]$.
Thus,  from the above estimates, it follows that 
\begin{align*}
& \|\big( g'(\nabla u_1) u_1' - a'(\nabla u_1) (u_1')^2  
- \big(g'(\nabla u_2) u_2' - a'(\nabla u_2) (u_2')^2\big) \big)(t) \Pi(\nabla X, \xi ) \|_{C^{2\a  -3}} \\
\lesssim &  K(\|{\bf u}_1 \|_{\a, \b, \ga}, \|{ \bf u}_2\|_{\a, \b, \ga})
\|{\bf u}_1 -{\bf u}_2\|_{\a, \b, \ga}
(1+\|X\|_{C^\a})^2\|\Pi(\nabla X, \xi )\|_{C^{2\a -3}}. 
\end{align*}
For the third term $\Pi(a(\nabla u_0^T), \bar\Pi_{u'}\xi )$, from \eqref{eq:2.18} and
$\a + \b -2 <0< \a + \gamma -3$, we have
\begin{align*}
& \|\Pi(a(\nabla u_0^T), \bar\Pi_{u_1'}\xi -\bar\Pi_{u_2'}\xi  ) (t)\|_{C^{\a + \b -2}} \\
& \qquad  \lesssim \|\Pi(a(\nabla u_0^T), \bar\Pi_{u_1'}\xi -\bar\Pi_{u_2'}\xi  )\|_{C_TC^{\a + \gamma -3}} \\
& \qquad \lesssim 
\|a\|_{C^1}(1+ T^{-\frac{\gamma-\a}{2}}\|  u_0\|_{C^\a} ) 
\|{\bf u}_1 -{\bf u}_2\|_{\a, \b, \ga}
 \|\xi\|_{C^{\alpha-2}}.
\end{align*}
Multiply the both sides by  $t^{\frac{\ga -\a}2}$ and take the supremum in $t\in (0,T]$.
Then the factor $T^{-\frac{\gamma-\a}{2}}$ (which is large for small $T>0$)
is absorbed and we obtain the desired
estimate for this term.

Consequently, the proof of \eqref{eq-4.2-0415}  is completed by 
recalling $\a+ \b -2< 2\a  -3$. 
\end{proof}
We proceed to  handle the term $A_1$ defined by \eqref{eq:A1} 
with $R_1 = R_1(u', X) = [\nabla,\bar\Pi_{u'}]X$.

\begin{lem}\label{lem-A1}
For any ${\bf u}_1 = (u_1,u_1'), {\bf u}_2 = (u_2,u_2') \in \mathcal{B}_T(\la)$,  
we have
\begin{align}
& \| A_1({\bf u}_1) \|_{C_TC^{2\a +\b -3}} \lesssim  
K( \|{\bf u}_1 \|_{\a, \b, \ga})(1+ \|\xi\|_{C^{\a-2}})^2 \|\xi\|_{C^{\a-2}},
\label{eq-4.4-0417} 
\\
 &  \| A_1({\bf u}_1) -A_1({\bf u}_2)\|_{C_TC^{2\a +\b -3}}   \label{eq-4.5-0417}
\\
\lesssim &
K( \|{\bf u}_1 \|_{\a, \b, \ga}, \|{\bf u}_2 \|_{\a, \b, \ga})\|{\bf u}_1-{\bf u}_2\|_{\a, \b, \ga} 
(1+\|\xi\|_{C^{\a-2}})^4. \notag
\end{align}
\end{lem}

\begin{proof}
We start with the proof of \eqref{eq-4.4-0417} .
Noting that $2\a-3<0$ and $3\a -4>0$, by Lemma 2.7 of \cite{GIP-15}, we have
$P_g:(\nabla u,\xi) \in C^{\a-1}\times C^{\a-2}
\mapsto P_g(\nabla u,\xi) \in C^{2(\a-1)+\a-2} = C^{3\a -4}$
and 
\begin{align}\label{eq-4.6-0417}
\|P_g(\nabla u_1,\xi)\|_{C_TC^{3\a -4}}\lesssim & \|g\|_{C^2}
(1+ \|\nabla u_1\|_{C_TC^{\a-1}}^2)\|\xi\|_{C^{\a-2}}\\
\lesssim & 
\|g\|_{C^2}
\big(1+ (1+ \|X\|_{C^\a})^2\|{\bf u}_1\|_{\a, \b, \ga}^2 \big) \|\xi\|_{C^{\a-2}}   \notag \\
\lesssim & \|g\|_{C^2}
(1+\|{\bf u}_1\|_{\a, \b, \ga}^2 ) (1+ \|X\|_{C^\a})^2 \|\xi\|_{C^{\a-2}}, \notag
\end{align}
where we have used  \eqref{eq-3.5-0321} for the second inequality.
On the other hand, thanks to  Lemmas \ref{lem modified commutator} and \ref{lem:37} together with \eqref{eq-3.4-0424}, we easily have that the 
$C_TC^{2\a +\b-3}$-norm of the second term of $A_1$
is bounded by 
$$
K(\|{\bf u}_1\|_{\a, \b, \ga}) (1+\|X\|_{C^\a})\| X\|_{C^\a}\|\xi\|_{C^{\a-2}}.
$$
Therefore, we obtain  \eqref{eq-4.4-0417} by the above estimates.

From now on, we show \eqref{eq-4.5-0417}. 
For the first term of $A_1$,  
by the Lipschitz estimate given in Lemma 2.7 of \cite{GIP-15} and \eqref{eq-3.5-0321}, we have
\begin{align} \label{eq-4.6-430}
& \|P_g(\nabla u_1,\xi) - P_g(\nabla u_2,\xi)\|_{C_TC^{3\a -4}}\\
 \lesssim &  \|g\|_{C_b^3} \Big(1+ \big(\|\nabla u_1\|_{C_TC^{\a-1}}+ \|\nabla u_2\|_{C_TC^{\a-1}}\big)^2
+\|\xi\|_{C^{\a-2}}\Big) \|\nabla u_1-\nabla u_2\|_{C_TC^{\a-1}}  \notag \\
 \lesssim & \|g\|_{C_b^3} 
\Big(1+ (1+\|X\|_{C^\a})^2\big(\|{\bf u}_1\|_{\a, \b, \ga}+ \|{\bf u}_2\|_{\a, \b, \ga}\big)^2
+\|\xi\|_{C^{\a-2}}\Big) (1+\|X\|_{C^\a})\|{\bf u}_1-{\bf u}_2\|_{\a,\b,\ga}, \notag
\end{align}
which gives the bound 
\begin{align*}
& \|P_g(\nabla u_1,\xi) - P_g(\nabla u_2,\xi)\|_{C_TC^{3\a -4}}
\\
\lesssim &
K( \|{\bf u}_1 \|_{\a, \b, \ga}, \|{\bf u}_2 \|_{\a, \b, \ga}) 
\|{\bf u}_1 - {\bf u}_2 \|_{\a, \b, \ga}(1+ \|X\|_{C^\a})^3 (1+\|\xi\|_{C^{\a-2}}).
\end{align*}  
Thanks to the multilinear properties of $\bar{C}$, the resonant and $R_1$,
we can easily show that
\begin{align*}
& \| g'(\nabla u_1) \big\{\bar C(u_1', \nabla X,\xi) 
 +\Pi(R_1(u_1',X),\xi)\big\} 
\\
& \qquad \qquad  -g'(\nabla u_2) \big\{\bar C(u_2', \nabla X,\xi) 
 +\Pi(R_1(u_2',X),\xi)\big\}\|_{C_TC^{2\a+ \b-3}}
\\
 \lesssim&  K( \|{\bf u}_1 \|_{\a, \b, \ga}, \|{\bf u}_2 
\|_{\a, \b, \ga})  \| {\bf u}_1-{\bf u}_2 \|_{\a, \b, \ga} (1+ \|X\|_{C^\a})^2\|X\|_{C^\a} \|\xi\|_{C^{\a-2}}.
\end{align*}
For instance, by the linearity of  $\bar{C}$ in $u'$  and Lemma \ref{lem modified commutator}, we have
\begin{align*} 
 \| \bar{C}(u_1',\nabla X,\xi) - \bar{C}(u_2',\nabla X,\xi) \|_{C_T C^{2\a+\b-3}}
= &\| \bar{C}(u_1' -u_2',\nabla X,\xi)  \|_{C_TC^{2\a+\b-3}} \\
\lesssim &
 \| {\bf u}_1-{\bf u}_2 \|_{\a, \b, \ga} \| X\|_{C^{\a}}
 \|\xi\|_{C^{\a-2}}.
\end{align*}
Then, thanks to \eqref{eq-3.4-0424} and \eqref{eq-3.8-03038},  the usual arguments give the desired result
for $g'(\nabla u_1)$ $\bar C(u_1',\nabla X,\xi)$.
 Consequently, the above estimates and the fact that $2\a +\b-3< 3\a -4$ yield \eqref{eq-4.5-0417}.
\end{proof}

For the reader's convenience, we recall the definition of $A_2$ from \eqref{eq:A2}:
\begin{align*}
A_2 =& A_2(u, u')=  -R(a(\nabla u)- a(\nabla u_0^T), u';\xi)  - \Pi_{\bar\Pi_{u'}\xi}
(a(\nabla u)- a(\nabla u_0^T))  \\
& \hskip 20mm  
+ (a(\nabla u)- a(\nabla u_0^T)) R_2,
\end{align*}
where $R$ is defined in Lemma  \ref{lem-2.4-0405} 
and $R_2=R_2(u', X)= [\De,\bar\Pi_{u'}]X$.
The next lemma gives the  estimates on $A_2$.

\begin{lem}\label{lem-A2}
For any ${\bf u}_1 = (u_1,u_1'), {\bf u}_2 = (u_2,u_2') \in \mathcal{B}_T(\la)$,  
we have
\begin{align}
& \| A_2({\bf u}_1) \|_{C_TC^{\a +\b -2}} \lesssim   
K( \|{\bf u}_1 \|_{\a, \b, \ga})(1+\|\xi\|_{C^{\a-2}}) \|\xi\|_{C^{\a-2}},
            \label{eq-4.14-0417} \\
& \| A_2({\bf u}_1) -A_2({\bf u}_2) \|_{C_TC^{\a +\b -2}}   \label{eq-4.15-0417}\\
 \lesssim &
K( \|{\bf u}_1 \|_{\a, \b, \ga}, \|{\bf u}_2 \|_{\a, \b, \ga})
\|{\bf u}_1- {\bf u}_2\|_{\a, \b, \ga}
(1+\|\xi\|_{C^{\a-2}})^2 \|\xi\|_{C^{\a-2}}. \notag
\end{align}
\end{lem}
\begin{proof}
Let us first give the proof of \eqref{eq-4.14-0417}.  For the first term in $A_2$, using  Lemma \ref{lem-2.4-0405} together with  \eqref{eq-5.8-0301}, we have 
\begin{align} \label{eq-4.18-0417}
 \|R( a(\nabla u_1)- a(\nabla u_0^T), u_1';\xi)\|_{C_TC^{ \a +\b -2}} 
\lesssim & 
\|a(\nabla u_1)- a(\nabla u_0^T)\|_{C_TC^{\b}} \|u_1'\|_{\mathcal{L}_T^\b}
\|\xi\|_{C^{\a-2}}
 \\
\lesssim &  K(\|{\bf u}_1\|_{\a, \b, \ga})
(1+\|X\|_{C^\a})\|\xi\|_{C^{\a -2}}.
\notag 
\end{align}
By Lemmas \ref{lem:bar} and  \ref{lem:37} together with  \eqref{eq-5.8-0301}, the similar arguments yield the same local growth of the other two terms as above.
So, we have \eqref{eq-4.14-0417}.

Next, we show the local Lipschitz estimate \eqref{eq-4.15-0417}. 
The proof is essentially same as that of \eqref{eq-4.14-0417}, 
because of the multilinear properties of $R$, the resonant, the modified paraproduct and $R_2$. Here we only give the proof for the term $R$ as an example. 
By the trilinearity of $R$, the same arguments for \eqref{eq-4.18-0417} give that
\begin{align*}
& \|R( a(\nabla u_1)- a(\nabla u_0^T), u_1';\xi)
-R(a(\nabla u_2)- a(\nabla u_0^T), u_2';\xi)
\|_{C_TC^{ \a+\b  -2}}
\\
\lesssim & \|R(a(\nabla u_1)-a(\nabla u_2), u_1';\xi)
\|_{C_TC^{ \a +\b -2 }}
+ \|R(a(\nabla u_2)- a(\nabla u_0^T), u_1'-u_2';\xi)
\|_{C_TC^{  \a +\b  -2}} \notag\\
\lesssim &  \big(
\|a(\nabla u_1)-a(\nabla u_2)\|_{C_TC^{\b}} \|u_1'\|_{\mathcal{L}_T^\b}  +
\|a(\nabla u_2)- a(\nabla u_0^T)\|_{C_TC^{\b}} \|u_1'-u_2'\|_{\mathcal{L}_T^\b} \big)
\|\xi\|_{C^{\a-2}}
\notag \\
\lesssim &  K(\|{\bf u}_1\|_{\a, \b, \ga}, \|{\bf u}_2 \|_{\a, \b, \ga})
 \|{\bf u}_1 -  {\bf u}_2 \|_{\a, \b, \ga} 
(1+\|X\|_{C^\a})^2\|\xi\|_{C^{\a -2}},
\notag 
\end{align*}
where we have used \eqref{eq-5.8-0301}  and \eqref{eq-3.13-0409}.
\end{proof}

Finally, let us give the estimates on $A_3$. Recall 
the definition \eqref{eq:A3} of $A_3$:
\begin{align*}
A_3
=   P_a(\nabla u, \bar\Pi_{u'}\xi)+ a'(\nabla u) \big\{\bar{C}(u',\nabla X,\bar\Pi_{u'}\xi)
+ \Pi(R_1, \bar\Pi_{u'}\xi)  +u'\bar{C}(u',\xi,\nabla X)  \big\}.
\end{align*}
\begin{lem}\label{lem-A3}
For any ${\bf u}_1 = (u_1,u_1'), {\bf u}_2 = (u_2,u_2') \in \mathcal{B}_T(\la)$,  
we have
\begin{align}
& \|A_3({\bf u}_1) \|_{C_TC^{2\a +\b -3}} \lesssim  
K( \|{\bf u}_1 \|_{\a, \b, \ga})
(1+ \|\xi\|_{C^{\a-2}})^2\|\xi\|_{C^{\a-2}},  \notag\\
&  \|A_3({\bf u}_1) -A_3({\bf u}_2) \|_{C_TC^{2\a +\b -3}} 
 \label{eq-4.26-0417} \\
\lesssim & 
K( \|{\bf u}_1 \|_{\a, \b, \ga}, \|{\bf u}_2 \|_{\a, \b, \ga})
\|{\bf u}_1-{\bf u}_2\|_{\a, \b, \ga} 
(1+ \|\xi\|_{C^{\a-2}})^4. \notag
\end{align}
\end{lem}
\begin{proof}
Let us first give the estimates on the term $A_{3,1}(u, u'):=a'(\nabla u) u'\bar{C}(u',\xi,\nabla X)$.
Noting that $\beta>2\alpha+\beta-3>0$ and using Lemma \ref{lem-5.1-0301}-(ii) and
 Lemma \ref{lem modified commutator}, we have that
\begin{align} \label{eq-4.13-501}
\| u_1'\bar{C}(u_2',\xi,\nabla X)\|_{C_TC^{2\alpha+\beta-3}}
\lesssim \|u_1'\|_{\mathcal{L}_T^\beta}\|u_2'\|_{\mathcal{L}_T^\beta} \|X\|_{C^{\alpha}}\|\xi\|_{C^{\alpha-2}}
\end{align}
holds for $u_1', u_2'\in \mathcal{L}_T^\beta$. 
Using  \eqref{eq-4.13-501} with $u_1'=u_2'$ and \eqref{eq-3.4-0424}, we get the local growth estimate on $A_{3,1}$, that is, 
\begin{align*}
\| A_{3,1}(u_1, u_1')\|_{C_TC^{2\alpha+\beta-3}} 
\lesssim
K(\|{\bf u}_1\|_{\a, \b, \ga})(1+\|X\|_{C^\a})\|X\|_{C^\a} \| \xi\|_{C^{\a-2}}.
\end{align*}
Moreover, noting  the multilinear  property of $\bar{C}$,  by \eqref{eq-4.13-501}, \eqref{eq-3.4-0424} and \eqref{eq-3.8-03038},
we immediately have the local Lipschitz estimate on $A_{3,1}$, that is, 
\begin{align*}
& \| A_{3,1}(u_1, u_1') -A_{3,1}(u_2, u_2')
\|_{C_TC^{2\alpha+\beta-3}}  \\
\lesssim & 
K(  \|{\bf u}_1\|_{\a, \b,\ga}, 
 \|{\bf u}_2\|_{\a, \b, \ga} \big)\|{\bf u}_1 -{\bf u}_2\|_{\a, \b, \ga}    (1+\|X\|_{C^\a})^2  \|X\|_{C^\a} \|\xi\|_{C^{\a-2} }.  
\end{align*}

Now comparing  the terms in $A_3$ except $A_{3,1}$ with  $A_1$, see \eqref{eq:A1}, we find the 
difference is that  $g'$ and $\xi$ in $A_1$ are replaced by $a'$ and 
$\bar\Pi_{u'}\xi $ in $A_3$, respectively. So,  noting  $\| \bar\Pi_{u'}\xi \|_{C_TC^{\a-2}}\lesssim \|u'\|_{\mathcal{L}_T^\b}\| \xi \|_{C^{\a-2}}$ 
from Lemma \ref{lem:bar}-{\rm (i)}, we can easily conclude the proof by mimicking 
that of  Lemma \ref{lem-A1}. Here, we only explain the factor 
$(1+ \|\xi\|_{C^{\a-2}})^4$ appearing in \eqref{eq-4.26-0417} in a little more detail.  It comes from  the estimate on $P_a$ as we saw in the proof of \eqref{eq-4.5-0417}. In fact, using the similar arguments for  \eqref{eq-4.6-430} and then   \eqref{eq-3.5-0321}, we easily have 
\begin{align*}
 & \|P_a(\nabla u_1, \bar\Pi_{u_1'}\xi) -P_a(\nabla u_2, \bar\Pi_{u_2'}\xi)\|_{C_TC^{3\a -4} } \\
\lesssim & \|a\|_{C^3} 
\Big(1+(\|\nabla u_1\|_{C_TC^{\a-1}} +\|\nabla u_2\|_{C_TC^{\a-1}})^2 
+\|\bar\Pi_{u_2'}\xi\|_{C_TC^{\a-2}}  \Big) \notag\\
&  \qquad \times 
\big( \|\nabla u_1 -\nabla u_2\|_{C_TC^{\a-1}} 
+\|\bar\Pi_{u_1'-u_2'}\xi \|_{C_TC^{\a-2}} \big)  \notag \\
\lesssim & \|a\|_{C^3} 
\Big(1+
(1+\|X\|_{C^\a})^2(\|{\bf u}_1\|_{\a, \b,\ga} +\|{\bf u}_2\|_{\a, \b,\ga})^2 
+\|{\bf u}_2\|_{\a, \b,\ga} \|\xi \|_{C^{\a-2}}  \Big) \notag\\
&  \qquad \times 
\|{\bf u}_1-{\bf u}_2\|_{\a, \b,\ga}\big( 1+\|X\|_{C^\a}
+ \|\xi\|_{C^{\a-2}} \big),\\
\lesssim & K(  \|{\bf u}_1\|_{\a, \b,\ga},
 \|{\bf u}_2\|_{\a, \b, \ga} ) \|{\bf u}_1 -{\bf u}_2\|_{\a, \b, \ga} \\
& \hskip 10mm \times  
(1+ \|X\|_{C^\a})^2(1+ \|\xi\|_{C^{\a-2}}) (1+ \|X\|_{C^\a} + \|\xi\|_{C^{\a-2}}), 
\end{align*}
which gives the desired result.
\end{proof}
To conclude this section, let us give the proof of Proposition \ref{lem:Lip-zeta}.
\begin{proof}[Proof of Proposition \ref{lem:Lip-zeta}]
Noting that $\a +\b-2< 2\a +\b -3$ and $\ga -\a>0$,
we obtain immediately Proposition \ref{lem:Lip-zeta} by Lemmas \ref{lem-A0}-\ref{lem-A3}.
\end{proof}

\section{Convergence of the resonant term}  \label{sec:Renorm}

Recall $\xi\in C^{\a-2}$ and $\nabla X\in C^{\a-1}$.  Then we can define
$\Pi(\nabla X,\xi)\in C^{2\a-3}$, which is denoted by $\nabla X \diamond \xi$
in Lemma \ref{lem:2.9}, though their product
is definable with less regularity: $\nabla X\cdot\xi \in C^{\a-2}$ as we will see
in Lemma 5.3, note that $2\a-3>\a-2$. 

We follow the arguments in Section 5.2 of \cite{GIP-15} noting that
they discuss two dimensional case taking $\T=[0,2\pi]$, while we are 
in one dimension but consider $\nabla X$ instead of $X (=\vartheta)$.
Recall $\T=[0,1]$ in our case and set $\hat u(k) = \int_\T e^{-2\pi ikx}u(x)dx,
k\in \Z$.  Let $\xi$ be the spatial white noise on $\T$.  Then,
$$
E[\hat\xi(k)\hat\xi(k')] = 1_{k=-k'}, \quad k, k'\in \Z,
$$
and $\overline{\hat\xi(k)} = \hat\xi(-k)$.  Note that the mean zero
solution $X$ of \eqref{eq:Xxi} is given by
\begin{equation}  \label{eq:4.X}
X= \int_0^\infty P_t Q\xi dt,
\end{equation}
where $P_t= e^{t\De}$ and $Q\xi = \xi- \hat\xi(0)$, note $\hat\xi(0)=\xi(\T)$.

The following expectation appears to compensate the 0th order term
to define $\nabla X \diamond \xi$ in \eqref{eq:5-diamond}
(cf.\ Lemma 5.6 of \cite{GIP-15}), but it vanishes in our case.

\begin{lem} \label{lem:2.8} 
For $x\in \T$ and $t>0$, we have
\begin{align*}
E[\Pi(\nabla P_tQ \xi,\xi)(x)] 
=  \sum_{k\in \Z\setminus \{0\}} 2\pi ik e^{-4\pi^2 k^2 t} =0.
\end{align*}
\end{lem}

\begin{proof}
Compared with Lemma 5.6 of \cite{GIP-15} (they consider on $\T^2=[0,2\pi]^2$
so that they have $(2\pi)^{-2}$), 
we have $\nabla$ which yields $2\pi ik$ in Fourier mode.
\end{proof}

We assume the following rather mild condition for the mollifier
$\psi$ to cover the noise in \eqref{eq:we}, see Remark \ref{rem:5.1} below.
Let $\psi$ be a measurable and integrable function on $\R$, which has a compact support
and satisfies $\int_\R \psi(x)dx =1$.  We set $\psi^\e(x) = \frac{1}{\e}\psi(\frac{x}\e)$
for $\e>0$.
Note that the support of $\psi^\e$ is included in $\T \, \big(\!=[-\frac12,\frac12)\big)$ for
sufficiently small $\e>0$ and $\psi^\e*\xi(x) =\int_\T\psi^\e(x-y)\xi(y)dy, x\in \T$ is
well-defined (by considering $\psi^\e$ periodically on $\R$ if necessary). 
The compact support property is assumed for simplicity and can be removed.
We call $\xi^\e:=\psi^\e*\xi$ the smeared noise of the spatial white noise $\xi$ on $\T$.

\begin{lem} \label{lem:2.9} (cf.\ Lemma 5.8 of \cite{GIP-15})
Set
\begin{equation}  \label{eq:5-diamond}
\nabla X \diamond \xi = \int_0^\infty  \Pi(\nabla P_tQ \xi,\xi) dt.
\end{equation}
Then we have
\begin{equation}  \label{eq:5.3-B}
E[\|\nabla X \diamond \xi\|_{C^{2\a-3} }^p]<\infty
\end{equation}
for all $\a<\frac32$ and $p\ge 1$.  Moreover, for $\psi$ satisfying the above condition,
we set $\xi^\e= \psi^\e*\xi$ for $\e>0$ and $\nabla X^\e = \int_0^\infty \nabla P_t Q\xi^\e dt
\, \big(= \nabla (-\De)^{-1}Q\xi^\e\big)$.
Then, we have
\begin{equation}  \label{eq:5.4-B}
\lim_{\e\downarrow 0} E[\|\nabla X \diamond \xi- \Pi(\nabla X^\e,\xi^\e)\|_{C^{2\a-3} }^p]=0
\end{equation}
for all $p\ge 1$.  We also have
\begin{align}  \label{eq:c-e}
c_\e :=& E[\nabla X^\e(x) \xi^\e(x)] = E[\Pi(\nabla X^\e, \xi^\e)(x)] \\
= &\int_0^\infty E[\Pi(\nabla P_tQ\xi^\e, \xi^\e)(x)] dt
= \sum_{k\in \Z\setminus\{0\}} \frac{|\hat\psi(\e k)|^2}{4\pi^2k^2} 2\pi ik 
\notag  \\
=& 0,  \notag
\end{align}
for $x\in \T$ and $\e>0$ such that supp$(\psi) \subset \{|x|\le \frac1{2\e}\}$,
where $\hat\psi(y) := \int_\R e^{-2\pi i yx}\psi(x)dx, y \in \R$
is the Fourier transform on $\R$.
\end{lem}

\begin{proof}
We first note that  $c_\e=0$ in \eqref{eq:c-e} follows from the symmetry of
$|\hat{\psi}(\varepsilon k)|^2=\hat{\psi}(\varepsilon k)\hat{\psi}(-\varepsilon k)$
in $k$ due to the fact that $\psi$ is real-valued.

To show \eqref{eq:5.3-B} and \eqref{eq:5.4-B},
we divide the time integral on $(0,\infty)$  in \eqref{eq:5-diamond} and $\nabla X^\e$
into those on $(0,1]$ and $(1,\infty)$.  Let us first show \eqref{eq:5.4-B} for 
the contribution from the integral on $(1,\infty)$, \eqref{eq:5.3-B} for this
part is shown similarly.  Noting that $2\a-3<0<2\a-2$ for 
$\a\in (\frac43,\frac32)$, we have
\begin{align*}
\Big\| & \int_1^\infty \big\{ \Pi(\nabla P_t Q\xi, \xi) - \Pi(\nabla P_t Q\xi^\e, \xi^\e)\big\} dt
\Big\|_{C^{2\a-3}}  \\
& \le \int_1^\infty \big\{ \|\Pi(\nabla P_t Q(\xi-\xi^\e), \xi)\|_{C^{2\a-2}} 
+ \|\Pi(\nabla P_t Q\xi^\e, \xi- \xi^\e)\|_{C^{2\a-2}} \big\} dt \\
& \lesssim \int_1^\infty \big\{ \|\nabla P_t Q(\xi-\xi^\e)\|_{C^\a} \| \xi\|_{C^{\a-2}}
+ \|\nabla P_t Q\xi^\e\|_{C^\a} \|\xi- \xi^\e\|_{C^{\a-2}} \big\} dt \\
& \lesssim \int_1^\infty \big\{ \|P_t Q(\xi-\xi^\e)\|_{C^{\a+1}} \| \xi\|_{C^{\a-2} }
+ \|P_t Q\xi^\e\|_{C^{\a+1}} \|\xi- \xi^\e\|_{C^{\a-2}} \big\} dt \\
& \lesssim  \|\xi-\xi^\e\|_{C^{\a-2}} \big( \| \xi\|_{C^{\a-2} }
+ \| \xi^\e\|_{C^{\a-2}} \big) \int_1^\infty t^{-\frac32}dt,
\end{align*}
where we have used Lemma \ref{lem-5.1-0301}-(i) (noting $2\a-2>0$)
for the second inequality, Lemma \ref{lem-2.3-0309}-(i) for the last inequality,
and note that $\int_1^\infty t^{-\frac32}dt<\infty$.  We then have
the desired convergence for the part arising from the integral on $(1,\infty)$ by showing
\begin{equation}  \label{eq:5.6-B}
E\big[\|\xi-\xi^\e\|_{C^{\a-2}}^p \| \xi^\e\|_{C^{\a-2}}^p \big]
\le E\big[\|\xi-\xi^\e\|_{C^{\a-2}}^{2p}\big]^{\frac12} 
\sup_{0<\e<1} E\big[ \| \xi^\e\|_{C^{\a-2}}^{2p} \big]^{\frac12} 
\underset{\e\downarrow 0}{\longrightarrow} 0,
\end{equation}
under our condition on $\psi$.  Indeed, we first compute
$$
E\big[\|\xi-\xi^\e\|_{B_{2p,2p}^{\a-2}}^{2p}\big]
= \sum_{j=-1}^\infty 2^{2pj(\a-2)} \int_\T dx \, E[|\De_j(\xi-\xi^\e)(x)|^{2p}].
$$
Then, by Gaussian hypercontractivity (equivalence of moments, 
Lemma 4.6 of \cite{GIP-15}),
$$
E[|\De_j(\xi-\xi^\e)(x)|^{2p}]
\le C_p E[\De_j(\xi-\xi^\e)(x)^2]^p,
$$
for some $C_p>0$.
Here, with $K_j := \check\rho_j$ (inverse Fourier transform of
dyadic partition $\{\rho_j\}_{j=-1}^\infty$ of unity), we can rewrite as
\begin{align*}
E[\De_j(\xi-\xi^\e)(x)^2]
&= E\Big[ \Big(\int_\T K_j(x-y)(\xi-\xi^\e)(y)dy\Big)^2\Big] \\
& = \|K_j-K_j*\psi^\e\|_{L^2(\T)}^2 \\
& = \|\rho_j-\rho_j \hat{\psi}^\e\|_{L^2(\Z)}^2 \\
&=  \sum_{k\in\Z} \rho_j(k)^2 \big\{ 1-2\,\text{Re}\,\hat{\psi}^\e(k)
+ |\hat{\psi}^\e(k)|^2\big\}.
\end{align*}
We have used Plancherel identity for the third equality.
However, by our condition on $\psi$,
\begin{align*}
|\hat{\psi}^\e(k)| 
&= \Big|\frac1\e\int_\T \psi\Big(\frac{x}\e\Big) e^{-2\pi ikx}dx \Big| \\
&= \Big|\int_\R \psi(y) e^{-2\pi ik\e y}dy \Big| \le \|\psi\|_{L^1(\R)}
\end{align*}
and this tends to $0$ as $\e\downarrow 0$ for each $k$.
Thus, since $\sum_{k\in \Z}\rho_j^2(k) \sim \int \rho^2(2^{-j}x)dx \sim 2^j$,
we can show by Lebesgue's convergence theorem that
$$
E\big[\|\xi-\xi^\e\|_{B_{2p,2p}^{\a-2}}^{2p}\big]
\underset{\e\downarrow 0}{\longrightarrow} 0,
$$
if $2(\a-2)+1<0$, that is $\a<\frac32$.  Since we have
continuous embedding $B_{p,p}^\a \subset B_{\infty,\infty}^{\a-\frac1p}=C^{\a-\frac1p}$
by Besov embedding theorem (Lemma A.2 of \cite{GIP-15} or Lemma 8 of \cite{GP}),
taking $p$ large, we see
\begin{equation*}
E\big[\|\xi-\xi^\e\|_{C^{\a-2}}^{2p}\big]
\underset{\e\downarrow 0}{\longrightarrow} 0,
\end{equation*}
if $\a<\frac32$ and this implies \eqref{eq:5.6-B}.

For the integral on $(0,1]$, since one can apply Gaussian hypercontractivity,
Lebesgue's convergence theorem and Besov embedding theorem as above, 
we only show \eqref{eq:5.3-B} for this part:
$$
E\Big[ \Big\|\int_0^1 \Xi_tdt \Big\|_{C^{2\a-3} }^p\Big] <\infty,
$$
where $\Xi_t = \Pi(\nabla P_t Q\xi, \xi)$.
To show this, noting that $\nabla$ yields $2\pi ik$ in Fourier mode, similarly to
Lemma 5.8 of \cite{GIP-15} and also as above, we have
\begin{align*}
{\rm Var} &(\De_m \Pi(\nabla P_t\xi,\xi)(x))  \\
= & \sum_{k_1\not=0, k_2} \sum_{|i-j|\le 1} \sum_{|i'-j'|\le 1}
\Big[ 1_{m\lesssim i} 1_{m\lesssim i'} \rho_m^2(k_1+k_2) \rho_i(k_1)\rho_j(k_2)
\rho_{i'}(k_1)\rho_{j'}(k_2) 4\pi^2k_1^2e^{-8\pi^2k_1^2t} \\
& \qquad\qquad +  1_{m\lesssim i} 1_{m\lesssim i'} \rho_m^2(k_1+k_2) \rho_i(k_1)\rho_j(k_2)
\rho_{i'}(k_2)\rho_{j'}(k_1) 4\pi^2k_1k_2e^{-4\pi^2k_1^2t-4\pi^2k_2^2t}
\Big].
\end{align*}
Then, noting that $k_1\in{\rm supp} (\rho_i)$ and $k_2\in{\rm supp} (\rho_j)$
with $|i-j|\le 1$ imply $k_1^2\sim 2^{2i}$ and $|k_1k_2|\sim 2^{2i}$, we have
\begin{align*}
{\rm Var}& (\De_m \Pi(\nabla P_t\xi, \xi)(x))  \\
& \lesssim \sum_{i,j,i',j'} 1_{m\lesssim i} 1_{i\sim j \sim i' \sim j'}
\sum_{k_1, k_2} 1_{{\rm supp} (\rho_m)}(k_1+k_2) 1_{{\rm supp} (\rho_i)}(k_1)
1_{{\rm supp} (\rho_j)}(k_2) 2^{2i} e^{-2tc 2^{2i}} \\
& \lesssim \sum_{i: i\gtrsim m} 2^i 2^m 2^{2i} e^{-2tc 2^{2i}} 
\lesssim \frac{2^m}{t^{3/2}} \sum_{i: i\gtrsim m}  e^{-tc 2^{2i}} 
\lesssim \frac{2^m}{t^{3/2}} e^{-tc 2^{2m}},
\end{align*}
where we used that $\sharp k_1 \le C2^i, \sharp k_2 \le  C2^m$
in the sum $\sum_{k_1,k_2}$ in the second line 
(instead of  $\sharp k_1 \le C2^{2i}, \sharp k_2 \le  C2^{2m}$ in two dimensional
case) and  $2^{3i} e^{-tc 2^{2i}} \le Ct^{-3/2}$ in the third line.

Thus, we obtain
\begin{align*}
E[\|\Xi_t\|_{B_{2p,2p}^{2\a-3}}]
& \lesssim \Big( \sum_{m\ge -1} 2^{(2\a-3)m 2p}
E[\|\De_m \Xi_t\|_{L^{2p}(\T)}^{2p}] \Big)^{1/2p} \\
& \lesssim t^{-3/4} \Big( \sum_{m\ge -1} 2^{(2\a-3)m 2p}
2^{mp} e^{-tcp 2^{2m}} \Big)^{1/2p} \\
& \lesssim t^{-3/4} \Big( \int_{-1}^\infty  (2^x)^{2p(2\a-\frac52)}
e^{-tcp (2^x)^2} dx\Big)^{1/2p}.
\end{align*}
The change of variables $y=\sqrt{t}2^x$ yields
\begin{align*}
E[\|\Xi_t\|_{B_{2p,2p}^{2\a-3}}]
& \lesssim t^{-3/4} \Big( t^{-p(2\a-\frac52)} \int_0^\infty  y^{2p(2\a-\frac52)-1}
e^{-cpy^2}dy \Big)^{1/2p}.
\end{align*}
If $\a>\frac54$, the integral in the right hand side is finite for all large $p$
and therefore $E[\|\Xi_t\|_{B_{2p,2p}^{2\a-3}}]\lesssim t^{-\frac34-\frac12 (2\a-\frac52)}
= t^{-\a+\frac12}$  so that $\int_0^1 E[\|\Xi_t\|_{B_{2p,2p}^{2\a-3}}]dt < \infty$
for all $\a<\frac32$.  
\end{proof}

The constant $c_\e$ usually diverges as $\e\downarrow 0$, and in such case
it is called the renormalization constant.  However, since $c_\e=0$ in our case,
our equation \eqref{eq:Q-ab} does not require any renormalization.

\begin{rem}  \label{rem:5.1}
{\rm (i)} Let us consider the noise $\dot{w}^\e(x)$ introduced in \eqref{eq:we}
with $\{w(x)\}_{x\in \T}$.
It has a representation: $\dot{w}^\e(x)= \psi^\e*\dot{w}(x)$ with
\begin{equation}  \label{eq:psi-ab}
\psi(x) = \frac1{a+b}1_{[-a,b]}(x),
\end{equation}
which satisfies the assumption of Lemma \ref{lem:2.9}.
Indeed, since $\psi^\e(x) = \frac1{(a+b)\e}1_{[-a\e,b\e]}(x)$, we see
\begin{align*}
\psi^\e*\dot{w}(x) &= \int \psi^\e(x-y) \dot{w}(dy)
= \frac1{(a+b)\e} \int_{x-b\e}^{x+a\e}\dot{w}(dy)
= \dot{w}^\e(x).
\end{align*}
{\rm (ii)} At least heuristically, $\xi=\dot{w}(x)$ and
$\nabla X= \nabla (-\De)^{-1} \dot{w} = (-\nabla)^{-1} \dot{w} = 
-w(x)$ is a periodic mean zero Brownian motion, that is
$$
w(x) = B(x)-xB(1)- \int_0^1\{B(y)-yB(1)\}dy,
\quad x \in \T \simeq [0,1),
$$
where $B$ is a standard Brownian motion.
\end{rem}

Finally, we note that the product $\nabla X \cdot \xi$ can be defined directly
for $\xi=\dot{w}(x)$ 
in a usual sense as a limit, but we see $\nabla X\cdot \xi\in C^{-\frac12-\de}$
for every $\de>0$, though $\Pi(\nabla X,\xi) \in C^{2\a-3}, \a<\frac32$.
In particular, we see again that we don't need any renormalization.

\begin{lem}  \label{lem:5.3-B}
Let $w(x)$ be as in Remark \ref{rem:5.1}-(ii)  and let $\xi(x) = \dot{w}(x)$.
Then, $\nabla X^\e(x) \cdot \xi^\e(x)$ converges to
$- \frac12 \nabla (w^2(x))$ as $\e\downarrow 0$ in
$C^{-\frac12-\de}$ for every $\de>0$.
\end{lem}

\begin{proof}
Recall that
\begin{align*}
\xi^\e(x) & = \int_\T \psi^\e(x-y)\xi(dy) = \int_\T \psi^\e(x-y)\dot{w}(y) dy \\
& = \int_\T \nabla_x \psi^\e(x-y) w(y) dy,
\end{align*}
since $-\nabla_y \psi^\e(x-y) = \nabla_x \psi^\e(x-y)$.  
On the other hand, 
let $\{\fa_0, \fa_{k,\pm}\in C^\infty(\T)\}_{k=1}^\infty$ be the eigenfunctions
of $-\De$:  $-\De \fa_{k,\pm}(x) = 4\pi^2k^2 \fa_{k,\pm}(x)$,
where
$$
\fa_{k,+}(x) = \sqrt{2} \sin 2\pi kx, \quad
\fa_{k,-}(x) = \sqrt{2} \cos 2\pi kx, \qquad x \in \T \simeq [0,1),
$$
and $\fa_0(x) \equiv 1$.  Then the equation $-\De X^\e = Q\xi^\e$ is solved as
$$
X^\e(x) = \sum_{k=1, \pm}^\infty \frac{\fa_{k,\pm}(x)}{4\pi^2 k^2} 
\int_\T \fa_{k,\pm}(y) \xi^\e(y)dy.
$$
Therefore, by integration by parts,
\begin{align*}
\nabla X^\e(x) & = \sum_{k=1, \pm}^\infty \frac{\fa'_{k,\pm}(x)}{4\pi^2 k^2} 
\int_\T \fa_{k,\pm}(y) dy \int_\T \nabla_y \psi^\e(y-z) w(z) dz  \\
& = - \int_\T \partial_y\partial_x Y_y(x) dy \int_\T \psi^\e(y-z) w(z) dz \\
& = - \int_\T \psi^\e(x-z) w(z) dz, 
\end{align*}
where 
$$
Y_y(x) = \sum_{k=1,\pm}^\infty \frac{\fa_{k,\pm}(x)}{4\pi^2 k^2} \fa_{k,\pm}(y) 
$$
is a (mean zero) solution of
$$
-\De Y_y(x) = \de_y(x)
$$
for each $y\in \T$, that is, $Y_y(x)$ is a Green function of $-\De$. 
Note that $\partial_y\partial_x Y_y(x)= \de_y(x)$.

From these, we have
$$
\nabla X^\e(x) \cdot \xi^\e(x) = -\frac12 \nabla \Big(\int_\T \psi^\e(x-y) w(y) dy \Big)^2.
$$
In particular, since
$
\int_\T \psi^\e(x-y) w(y) dy
$
converges to $w(x)$ in $C^{\frac12-\de}$ as $\e\downarrow 0$
for every $\de>0$, we obtain the conclusion.
\end{proof}

\section{Comparison theorem for SPDE \eqref{eq:2} with smooth noise}
\label{sec:comp}

We show a comparison theorem for \eqref{eq:2} on $\T$ (or $\R$)
with smooth $\xi$.

\begin{lem}  \label{lem:comparison}  
{\rm (1)}  Assume $\chi\in C^1(\R)$ and satisfy $|\chi'(v)|\le C \fa'(v),
v\in \R$ for some $C>0$.  We also assume $\xi\in C^\infty(\T)$.  Then, for
two solutions $v_1, v_2$ of \eqref{eq:2} on $\T$, if $v_1(0)\ge v_2(0)$ holds, 
we have $v_1(t) \ge v_2(t)$ for all $t\ge 0$, where $v_1\ge v_2$ means that 
$v_1(x)\ge v_2(x)$ for all $x\in \T$ for $v_i=(v_i(x))_{x\in\T}, i=1,2$.    \\
{\rm (2)}  In addition, assume $\chi(0)=0$.  Then, if $v(0)\ge 0$, we have
$v(t)\ge 0$ for all $t\ge 0$.
\end{lem}

\begin{proof}
The assertion (2) follows from (1), since $v(t)\equiv 0$ is a solution of 
\eqref{eq:2} by noting $\chi(0)=0$. To show (1), assume that
$v_1(s,x)\ge v_2(s,x)$ for all $0\le s \le t$ and $x\in \T$ and
$v_1(t,x_0)= v_2(t,x_0)$ at some $t\ge 0$ and $x_0\in \T$.  Then, noting
that the solutions $v_1$ and $v_2$ of \eqref{eq:2} are smooth, we have
\begin{align*}
\partial_t &\big( v_1(t,x_0)-v_2(t,x_0)\big) \\
& = \De \big\{\fa(v_1(t,\cdot))-\fa(v_2(t,\cdot))\big\}(x_0)
+ \nabla\big\{ \big(\chi(v_1(t,\cdot))-\chi(v_2(t,\cdot))\big) \xi(\cdot) \big\}(x_0)  \\
& = \lim_{\de\downarrow 0} \Big[ \frac1{\de^2} \sum_\pm
 \big\{\fa(v_1(t,x_0\pm \de))-\fa(v_2(t,x_0\pm \de))\big\}\\
& \qquad\qquad  + \frac1\de \big\{ \big(\chi(v_1(t,x_0+\de))-\chi(v_2(t,x_0+\de))\big) 
\xi(x_0+\de) \big\} \Big].
\end{align*}
However, since $|\chi'(v)| \le C \fa'(v)$ and $|\xi(x)|\le M$ for some $M>0$, we see
\begin{align*}
\frac1\de &\big|\chi(v_1(t,x_0+\de))-\chi(v_2(t,x_0+\de))\big|\, \big|\xi(x_0+\de) \big|\\
&\le \frac{M}\de \int_{v_2(t,x_0+\de)}^{v_1(t,x_0+\de)} |\chi'(v)| dv  \\
& \le \de CM \cdot \frac1{\de^2} \big\{\fa(v_1(t,x_0+ \de))-\fa(v_2(t,x_0+ \de))\big\},
\end{align*}
which implies 
$$
\partial_t \big( v_1(t,x_0)-v_2(t,x_0)\big) \ge 0.
$$
This shows that $v_2$ cannot exceed $v_1$ at $t$ and $x_0$.
\end{proof}

\section*{Acknowledgements}
T.\ Funaki was supported in part by JSPS KAKENHI, Grant-in-Aid for Scientific Researches (A) 18H03672 and (S) 16H06338.
M.\ Hoshino was supported in part by JSPS KAKENHI, Early-Career Scientists 19K14556.
S.\ Sethuraman was supported by grant ARO W911NF-181-0311, 
a Simons Foundation Sabbatical grant, and by a Japan Society for the Promotion of Science Fellowship.
B.\ Xie was supported in part by JSPS KAKENHI, Grant-in-Aid for Scientific Research (C) 16K05197 and  (C) 20K03627.


\begin{thebibliography}{99}

\bibitem{BCD} {\sc H.\ Bahouri, J.-Y.\ Chemin and R.\ Danchin}, 
{\it Fourier Analysis and Nonlinear Partial Differential Equations},
Springer, 2011.

\bibitem{BB-16}{\sc I.\ Bailleul and F.\ Bernicot}, {\it Heat semigroup and singular PDEs}, J.\ Funct.\ Analysis, {\bf 270} (9) (2016), 3344--3452.


\bibitem{BDH-19}{\sc I.\  Bailleul, A.\ Debussche and M.\ Hofmanov\'{a}},
{\it Quasilinear generalized parabolic Anderson model equation},
Stoch.\ Partial Differ.\ Equ.\ Anal.\ Comput., {\bf 7} (2019), 40--63.

\bibitem{CC}{\sc G.\ Cannizzaro and K.\ Chouk},
{\it Multidimensional SDEs with singular drift
and universal construction of the polymer
measure with white noise potential},
Ann.\ Probab., {\bf 46} (2018), 1710--1763.

\bibitem{DaZ} {\sc G.\ Da Prato and J.\  Zabczyk}, 
{\it Ergodicity for Infinite-dimensional Systems}. 
London Mathematical Society Lecture Note Series, {\bf 229}. Cambridge University Press, Cambridge, 1996.

\bibitem{FIR}{\sc F.\ Flandoli, E.\ Issoglio and F.\ Russo},
{\it Multidimensional stochastic differential equations with
distributional drift},
Trans.\ Amer.\ Math.\ Soc., {\bf 369} (2017), 1665--1688. 

\bibitem{FS} {\sc  T.\ Funaki and H.\ Spohn},
{\it  Motion by mean curvature from the Ginzburg-Landau
$\nabla \phi$ interface model},
Comm.\ Math.\ Phys., {\bf 185} (1997), 1--36.

\bibitem{FuMa-19}{\sc  M.\ Furlan and M.\ Gubinelli},  
{\it Paracontrolled quasilinear SPDEs}, Ann.\ Probab., {\bf 47}
(2019), 1096--1135.

\bibitem{GIP-15}{\sc M.\ Gubinelli, P.\ Imkeller and N.\ Perkowski},
{\it Paracontrolled distributions and singular PDEs},
Forum of Mathematics, Pi 3, no. e6, (2015), 1--75.

\bibitem{GP}{\sc M.\ Gubinelli and N.\ Perkowski},
{\it Lectures on singular stochastic PDEs},
Ensaios Matem\'aticos [Mathematical Surveys], {\bf 29},
Sociedade Brasileira de Matem\'atica, Rio de Janeiro, 2015, 89 pp.

\bibitem{GP-17}{\sc M.\ Gubinelli and N.\ Perkowski}, 
{\it `KPZ reloaded'}, Comm.\ Math.\ Phys., {\bf 349} (1) (2017), 165--269.

\bibitem{Hai-14}{\sc M.\ Hairer}, A theory of regularity structures, Invent.\ Math., {\bf 198} (2014), 269--504.

\bibitem{HaLa-15}{\sc M.\ Hairer and C.\ Labb\'{e}}, 
{\it A simple construction of the continuum parabolic Anderson model on $R^2$}, Electron.\ Commun.\ Probab., {\bf 20} (2015), no. 43, 11 pp. 

\bibitem{HaLa-18}{\sc M. Hairer and C. Labb\'{e}},
{\it Multiplicative stochastic heat equations on the whole space},
J.\ Eur.\ Math.\ Soc., {\bf 20} (2018), 1005--1054.

\bibitem{Ho}{\sc M.\ Hoshino},
{\it Paracontrolled calculus and Funaki-Quastel approximation
for the KPZ equation},
Stoch.\ Proc.\ Appl., {\bf 128} (2018), 1238--1293.

\bibitem{KL}{\sc C.\ Kipnis and C.\ Landim},
{\it Scaling Limits of Interacting Particle Systems},
Grundlehren der Mathematischen Wissenschaften, {\bf 320},
Springer, 1999, xvi+442 pp. 

\bibitem{LSX}{\sc C.\ Landim, C.\ G.\  Pacheco, S.\ Sethuraman and J.\ Xue},
{\it On hydrodynamic limits in Sinai-type random environments},
in preparation, 2020.

\end{thebibliography}
\end{document}